\documentclass[11pt]{article}
\usepackage{fullpage}
\usepackage{authblk}
\usepackage{amsmath}
\usepackage{amsthm}
\usepackage{amssymb}
\usepackage{bm}
\usepackage{setspace}
\usepackage[numbers,sort&compress]{natbib}
\usepackage{color}

\usepackage{amsfonts}
\usepackage{eucal}
\usepackage{latexsym}
\usepackage{mathdots}
\usepackage{mathrsfs}
\usepackage{enumitem}

\newcommand{\be}{\begin{equation}}
\newcommand{\ee}{\end{equation}}
\newcommand{\ba}{\begin{eqnarray}}
\newcommand{\ea}{\end{eqnarray}}
\newcommand{\bal}{\begin{align}}
\newcommand{\eal}{\end{align}}
\newcommand{\baln}{\begin{align*}}
\newcommand{\ealn}{\end{align*}}
\newcommand{\bi}{\begin{itemize}}
\newcommand{\ei}{\end{itemize}}
\newcommand{\bn}{\begin{enumerate}}
\newcommand{\en}{\end{enumerate}}
\newcommand{\bbm}{\begin{bmatrix}}
\newcommand{\ebm}{\end{bmatrix}}
\newcommand{\bpm}{\begin{pmatrix}}
\newcommand{\epm}{\end{pmatrix}}
\newcommand{\bsm}{\left ( \begin{smallmatrix}}
\newcommand{\esm}{\end{smallmatrix} \right) }
\newcommand{\bp}{\begin{proof}}
\newcommand{\ep}{\end{proof}}
\newcommand{\nn}{\nonumber}

\newcommand{\mr}{\ensuremath{\mathrm}}
\newcommand{\scr}{\ensuremath{\mathscr}}
\newcommand{\mbf}{\ensuremath{\boldsymbol}}
\newcommand{\mc}{\ensuremath{\mathcal}}
\newcommand{\mf}{\ensuremath{\mathfrak}}
\newcommand{\ov}{\ensuremath{\overline}}
\newcommand{\sm}{\ensuremath{\setminus}}
\newcommand{\wt}{\ensuremath{\widetilde}}

\newcommand{\ga}{\ensuremath{\gamma}}

\newcommand{\la}{\ensuremath{\lambda }}
\newcommand{\om}{\ensuremath{\omega}}

\def\C{\mathbb{C}}

\def\D{\mathbb{D}}
\def\Z{\mathbb{Z}}
\def\N{\mathbb{N}}
\def\B{\mathbb{B}}
\def\ofr{\overline{\mathfrak{r}}}
\def\fr{\mathfrak{r}}
\def\fz{\mathfrak{z}}
\def\cc{\mathfrak{C}}
\def\bH{\mathbb{H}}
\def\wvec{\overrightarrow}

\def\A{\mathcal{A} _d}

\def\fp{\mathbb{C} \{ \mathbb{\mathfrak{z}} \} }

\newcommand{\cH}{\ensuremath{\mathcal{H}}}
\newcommand{\cA}{\ensuremath{\mathcal{A}}}

\newcommand{\fs}{\ensuremath{\sigma_{\mathbb{N}}}} %finite spectrum
\newcommand{\GL}{\ensuremath{\mathrm{GL}}} % general linear group
\newcommand{\SF}{\ensuremath{\mathrm{SF}}} % semi-Fredholm
\newcommand{\ind}{\ensuremath{\mathrm{ind}}} %index of an SF oeprator

\newcommand{\F}{\ensuremath{\mathbb{F} }}
\newcommand{\fB}{\ensuremath{\mathbb{B}}}

\newcommand{\Addresses}{{% additional braces for segregating \footnotesize
  \bigskip
  \footnotesize

  M.~T.~Jury, \textsc{Department of Mathematics, University of Florida}\par\nopagebreak
  \textit{E-mail address:} \texttt{mjury@ad.ufl.edu}

  \medskip

  R.~T.~W.~Martin, \textsc{Department of Mathematics, University of Manitoba}\par\nopagebreak
  \textit{E-mail address:} \texttt{Robert.Martin@umanitoba.ca}
    
\medskip

E.~Shamovich, \textsc{Department of Mathematics, Ben-Gurion University of the Negev}\par\nopagebreak
\textit{E-mail address:} \texttt{shamovic@bgu.ac.il}

\vspace{1cm}

}}

\newcommand{\ip}[2]{\ensuremath{\langle {#1} , {#2} \rangle}}
\newcommand{\ipcn}[2]{\ensuremath{\left( {#1} , {#2} \right) _{\C ^n}}}
\newcommand{\ipcN}[2]{\ensuremath{\left( {#1} , {#2} \right) _{\C ^N}}}

\newcommand{\dom}[1]{\ensuremath{\mathrm{Dom} ({#1}) }}
\newcommand{\sing}[1]{\ensuremath{\mathrm{Sing} ({#1}) }}
\renewcommand{\dim}[1]{\ensuremath{\mathrm{dim} \left( {#1} \right) }}
\newcommand{\ran}[1]{\ensuremath{\mathrm{Ran} \left( {#1} \right) }}

\renewcommand{\ker}[1]{\ensuremath{\mathrm{Ker} \left( {#1} \right) }}
\newcommand{\spr}[1]{\ensuremath{\mathrm{spr} \left( {#1} \right) }}

\newcommand{\re}[1]{\ensuremath{\mathrm{Re} \left( {#1} \right) }}

\numberwithin{equation}{section}
\numberwithin{subsection}{section}

\newtheorem{thm}[subsection]{Theorem}

\newtheorem{lemma}[subsection]{Lemma}
\newtheorem{prop}[subsection]{Proposition}
\newtheorem{cor}[subsection]{Corollary}
\newtheorem*{thm*}{Theorem}

\newtheorem{thmA}{Theorem}

\theoremstyle{definition}
\newtheorem{defn}[subsection]{Definition}
\newtheorem{remark}[subsection]{Remark}
\newtheorem{eg}[subsection]{Example}

\title{Non-commutative rational functions in the full Fock space}

\author{Michael T. Jury\thanks{Supported by NSF grant DMS-1900364}}
\affil[1]{\footnotesize University of Florida}

\author[2]{Robert T.W. Martin\thanks{Supported by NSERC grant 2020-05683}}
\affil[2]{\footnotesize University of Manitoba}

\author[3]{Eli Shamovich}
\affil[3]{ \footnotesize Ben-Gurion University of the Negev}

\date{}

\begin{document}
\maketitle

\begin{abstract}
A rational function belongs to the Hardy space, $H^2$, of square-summable power series if and only if it is bounded in the complex unit disk. Any such rational function is necessarily analytic in a disk of radius greater than one. The inner-outer factorization of a rational function, $\mf{r} \in H^2$ is particularly simple: The inner factor of $\mf{r}$ is a (finite) Blaschke product and (hence) both the inner and outer factors are again rational. 

We extend these and other basic facts on rational functions in $H^2$ to the full Fock space over $\C ^d$, identified as the \emph{non-commutative (NC) Hardy space} of square-summable power series in several NC variables. In particular, we characterize when an NC rational function belongs to the Fock space, we prove analogues of classical results for inner-outer factorizations of NC rational functions and NC polynomials, and we obtain spectral results for NC rational multipliers.   
\end{abstract}

\section{Introduction}

A rational function, $\mf{r}$, in the complex plane, is bounded in the unit disk, $\D$, if and only if it is analytic in a disk $r \D$ of radius $r>1$. Moreover, a rational function belongs to the Hardy space $H^2$ of square-summable Taylor series in the unit disk, if and only if is uniformly bounded in $\D$, \emph{i.e.} if and only if it belongs to $H ^\infty$, the algebra of uniformly bounded analytic functions in the disk. Recall that $H^\infty$ can be viewed as the \emph{multiplier algebra} of $H^2$, the algebra of all functions which multiply $H^2$ into itself. That is, any $h \in H^\infty$ defines a bounded linear multiplication operator $M_h : H^2 \rightarrow H^2$, where $M_h f := hf$ for $f \in H^2$. Such a multiplier is called \emph{inner} if $M_h$ is an isometry and \emph{outer} if $M_h$ has dense range. An element $f \in H^2$ is called outer if it is cyclic for the \emph{shift}, $S = M_z$. By classical results of Herglotz and F. Riesz, any element, $f$, of the Hardy spaces $H^2$ or $H^\infty$ has a unique inner-outer factorization, and the inner factor can be further factored into a \emph{Blaschke inner} which contains all the zero or `vanishing information' of $f$, and a \emph{singular inner} which has no zeroes in $\D$. The inner-outer factorization of a rational $\fr \in H^2$ is particularly simple: The inner factor of a $\fr$ is always a finite Blaschke product (this is again a rational function) so that the outer factor is also a rational function. Moreover, in the case where $\fr = p$ is a polynomial, its outer factor is a polynomial of degree not exceeding that of $p$. 

A canonical noncommutative (NC) analog of the classical $H^2$ is then the full Fock space, $\bH ^2 _d$, of square-summable power series in several non-commuting variables: Given a $d-$tuple of formal NC variables, $\mf{z} := ( \mf{z} _1 , \cdots , \mf{z} _d )$, any $f \in \bH ^2 _d$ is a power series
$$ f (\mf{z} ) = \sum _{\alpha \in \F ^d} \hat{f} _\alpha \mf{z} ^\alpha; \quad \quad \sum _\alpha | \hat{f} _\alpha | ^2 < + \infty, $$ where $\F ^d$ is the \emph{free monoid}, the set of all words in the $d$ letters $\{ 1 , \cdots , d \}$, and given any word $\alpha = i_1 \cdots i_n \in \F ^d$, $i_k \in \{ 1 , \cdots , d \}$, the free monomial, $\mf{z} ^\alpha$, is defined as $\mf{z} ^\alpha := \mf{z} _{i_1} \cdots \mf{z} _{i_n}$. (Multiplication in $\F ^d$ is defined by concatenation of words, the unit is $\emptyset$, the \emph{empty word}, containing no letters and $\mf{z} ^\emptyset :=1$. The free monoid is the universal monoid on $d$ generators.)  Left multiplication by any of the $d$ independent variables defines an isometry on $\bH ^2 _d$, $L_k := M^L _{\mf{z} _k}$. It is straightforward to check that the $L_k$ have pairwise orthogonal ranges so that the row $d-$tuple $L:= \left( L_1 , \cdots , L_d \right) : \bH ^2 _d \otimes \C ^d \rightarrow \bH ^2 _d$ defines an isometry from several copies of the Fock space into itself which we call the \emph{left free shift}.

The NC analog of $H^{\infty}$ is the $WOT-$closed algebra $\bH ^\infty_d$ generated by the left free shifts. Just as $H^2$ is a reproducing kernel Hilbert space, $\bH ^2_d$ is an NC reproducing kernel Hilbert space (NC-RKHS) in the sense of Ball, Marx, and Vinnikov \cite{BMV}.  That is, for any $Z \in \B ^d _n$ and $y,v \in \C ^n$, the \emph{matrix-entry point evaluation} $\ell _{Z,y,v} : \bH ^2 _d \rightarrow \C$ defined by 
$$ \ell _{Z,y,v} (f) = y^* f(Z) v, $$ is a bounded linear functional on $\bH ^2 _d$ so that by Riesz representation there is a unique \emph{NC Szeg\"o kernel vector}, $K\{Z , y ,v \}$ implementing this linear functional, 
$$ y^* f(Z) v = \ip{K\{ Z, y ,v \}}{f}_{\bH ^2 _d}. $$ With this identification, $\bH ^\infty_d$ is the algebra of left multipliers of $\bH ^2_d$. Additionally, the elements of $\bH ^2_d$ and $\bH ^\infty_d$ are NC functions in the sense of \cite{AM15,KVV} in the unit free row-ball: 
\[
\B^d_{\N} = \bigsqcup_{n =1} ^\infty  \B ^d _n; \quad \quad \B ^d _n := \left\{ \left( X_1,\cdots ,X_d \right) \in \C ^{n\times n} \otimes \C ^{1\times d} \left| \ \sum_{j=1}^d X_j X_j^* < I_n \right. \right\}.
\]
Each locally bounded (hence analytic \cite[Chapter VII]{KVV}) NC function in $\B^d_{\N}$ admits a Taylor series expansion in NC variables around the origin that converges in an appropriate sense in $\B^d_{\N}$. The functions in $\bH ^2 _d$ are those locally bounded NC functions in $\B ^d _\N$ square-summable Taylor coefficients and the functions in $\bH ^\infty_d$ are uniformly bounded in $\B^d_{\N}$.  

A non-commutative inner-outer factorization for elements of $\bH ^2_d$ and $\bH ^\infty_d$ was first obtained by Popescu \cite{Pop-vN, Pop-multi} and independently by Davidson-Pitts \cite{DP-inv}. Here, inner and outer in this NC setting are defined in direct analogy to classical theory: An NC function, $F$, in $\bH ^\infty _d$ is \emph{inner} if it defines an isometric left multiplier of Fock space and \emph{outer} if $M^L _F$ has dense range. An element $f \in \bH ^2 _d$ is \emph{outer} if it is cyclic for the isometric right shifts, $R_k := M^R _{\mf{z} _k}$. In \cite{NC-BSO}, the authors have refined this NC inner-outer factorization by extending the classical Blaschke-Singular-Outer factorization to the NC Hardy spaces $\bH ^\infty _d$ and $\bH ^2 _d$. Namely, any NC inner $\Theta \in \bH ^\infty _d$ further factors uniquely as the product of an NC Blaschke inner and an NC singular inner. As in the classical case, the NC Blaschke factor, $B$, of $f \in \bH ^2 _d$ encodes all information about the ``zeroes'' of $f$ and the NC singular inner factor is pointwise invertible in $\B^d_{\N}$. Several questions regarding this NC Blaschke-Singular-Outer factorization have remained open. In particular, a natural question is whether, in analogy with classical Hardy space theory, the inner factors of rational functions in $\bH ^2 _d$ are Blaschke. In this paper, we provide an affirmative answer to this question, we characterize when an NC rational function belongs to Fock space, and we provide applications to the spectral theory of multipliers of Fock space. 

A complex NC rational expression is any syntactically valid combination of the several NC variables $\mf{z} _1 , \cdots , \mf{z} _d$, the complex scalars, $\C$, the operations $+ , \cdot , ^{-1}$, and parentheses $\left( , \right)$ \cite{Volcic}. For example,
$$ \mr{r} (\mf{z} _1 , \mf{z} _2 ) = \mf{z} _1 ^3 \left( \mf{z}_2 \mf{z} _1 ^{-1} + 2 \right) ^{-1} -7 \mf{z}  _2 \mf{z} _1 \mf{z} _2.$$ The domain, $\dom{\mr{r}}$, of such an expression is simply the collection of all $d$-tuples of matrices of all sizes, $X = (X _1 , \cdots , X_d ) \in \C ^{n\times n} \otimes \C ^{1 \times d}$, $n \in \N$, for which $\mr{r} (X) \in \C ^{n \times n}$ is defined. An NC rational function, $\fr$, is an equivalence class of NC rational expressions with respect to the relation $\mr{r} _1 \equiv \mr{r} _2$, if $\dom{\mr{r} _1} \cap \dom{ \mr{r} _2} \neq \emptyset$ and for every $X$ in the intersection of their domains $\mr{r} _1(X) = \mr{r}_2(X)$. For an NC rational function $\mf{r}$, we will abuse notations and write $\dom{\mf{r}}$ for the collection of all points for which some $\mr{r} \in \mf{r}$ is defined. That is, $\dom{\fr} := \cup _{\mr{r} \in \fr} \dom{\mr{r}}$ and we write $\fr (X) := \mr{r} (X)$ for any $\mr{r} \in \fr$ with $X \in \dom{\mr{r}}$. By realization theory for NC rational functions, any NC rational function in $d-$variables, $\mf{r}$, with $0 \in \dom{\mf{r}}$ has a \emph{minimal realization}: There is a triple $(A, b, c)$ with $A \in \C ^{(n\times n) \cdot d} := \C ^{n \times n} \otimes \C ^{1\times d}$ and $b,c \in \C ^n$, so that for any $X \in \C ^{(m\times m) \cdot d}$,
$$ \mf{r} (X) = b^* L_A (X) ^{-1} c; \quad \quad L_A (X) := I - \sum A_j \otimes X_j.$$ (Here minimal means that $n$ is as small as possible \cite{KVV-ncrat,Volcic}, and $L_A $ is called a (monic) \emph{linear pencil}.) Realizations of rational functions have been studied extensively and have numerous applications to fields such as free probability and free real algebraic geometry \cite{DDSS,HKM-relax,HKMS,HMS-realize}. One of the main results of this paper is:

\begin{thmA} \label{thmA}
Let $\mf{r}$ be an NC rational function with minimal realization $(A,b,c)$ of size $N$. Then the following are equivalent.
\bn

\item[(i)] $\mf{r} \in \bH ^2_d$.

\item[(ii)] $\mf{r} = K \{ Z, y , v \}$ is an NC Szeg\"o kernel vector for some $Z \in \B ^d _N$ and $y, v \in \C ^N$.

\item[(iii)] The joint spectral radius, $\mr{spr} (A)$, of $A$ is $<1$. 

\item[(iv)] There exists $r > 1$, such that $r \B^d_{\N} \subset \dom{\mf{r}}$.

\item[(v)]  $\overline{\B^d_{\N}} \subset \dom{\mf{r} }$.

\item[(vi)] $\mf{r} \in \cA_d := \mr{Alg} (I, L) ^{- \| \cdot \| }$, the \emph{NC disk algebra}.

\item[(vii)]  $\mf{r} \in \bH ^\infty _d$. 
\en
\end{thmA}

For rational functions, $\fr (z)$, of a single complex variable, it is not hard to prove (using the Plancherel theorem) that if $\mf{r}$ is regular near $0$ and its power series at $0$ has square-summable coefficients,  then it cannot have any poles in the closed unit disk $|z|\leq 1$, and is therefore regular in a disk of radius $\rho>1$ (and hence bounded and continuous in $|z|\leq 1$). The equivlance of (i) with (iv)--(vii) in Theorem~\ref{thmA} can then be read as an extension of these facts to NC rational functions in the row ball (though the proofs will be rather more involved).  This is made more interesting by the observation that these one-variable facts do not generalize to the NC polydisk, for example (an example is given after the proof of Theorem~\ref{thmA}).

We will apply Theorem~\ref{thmA} to the study the spectra and cyclicity of NC rational multipliers. In Subsection \ref{sec:spectrum} we prove that the spectrum of a rational multiplier is determined by the spectra of its evaluations at finite levels of the free row-ball. Combined with results of Davidson-Pitts and Conway-Morrell we obtain:

\begin{thmA} \label{thmB}
Let $\fr $ be an NC rational function bounded on $\B^d_{\N}$. Let $\fr (L) := M^L _{\fr}$. Then $\fr (L)$ is a point of continuity of the spectral map $T \mapsto \sigma (T)$, $T \in \scr{L} (\bH ^2 _d )$.
\end{thmA}

As a consequence of the identifcation of NC rational functions in $\bH ^2 _d$ and NC kernels in Theorem \ref{thmA} above, we obtain that the inner and outer factors of any NC rational function in $\bH ^2 _d$ are rational, and that any rational NC inner function is Blaschke, see Section \ref{ratinout}. That is, NC rational functions in Fock space have no singular inner factor. Sections \ref{ratinout} and \ref{ncpoly} contain further, more detailed results on the inner-outer factorization of NC rational functions and polynomials in the Fock space.  In particular we highlight the following, which is an immediate consequence of Corollary~\ref{NCratouter}:

\begin{thmA}\label{thmC}
  An NC rational function $\fr(\mf{z})\in \bH^2_d$ is outer if and only if $\det \fr(Z)\neq 0$ for all $Z$ in the row ball.
\end{thmA}
In the classical one-variable Hardy space $H^2$ in the unit disk, it is obvious that a rational function in the space is cyclic for the unilateral shift if and only if it has no zeroes in $|z|<1$. The theorem just stated can be read as a generalization of this fact (though much less obvious). 

\section{Preliminaries}

Consider the \emph{complex NC universe}, 
$$ \C ^d _\N := \bigsqcup _{n=1} ^\infty \C ^{(n\times n) \cdot d}; \quad \C ^{(n\times n) \cdot d} = \C ^{n\times n} \otimes \C ^{1 \times d}, $$ the disjoint union of all $d-$tuples $X = \left( X _1 , \cdots , X_d \right) $ of $n \times n$ matrices of all finite sizes.  The NC row-ball $\B ^d _\N$ is the unit ball of $\C ^d _\N$ with respect to the row operator space norm on $\C ^d$. It will be occasionally be useful to include the infinite level $\B ^d _\infty$, and to consider the operator unit row-ball:
$$ \B ^d _{\aleph _0} := \B ^d _\N \sqcup \B ^d _\infty. $$ The infinite level is the set of all strict row contractions on a fixed separable Hilbert space.

As described in the introduction, the Fock space, $\bH ^2 _d$, can be defined as a Hilbert space of square-summable power series (as in \cite{Pop-freeholo}) which define locally bounded, hence analytic free NC functions in $\B ^d _\N$ or $\B ^d _{\aleph _0}$ \cite{KVV}. (A free NC function on an NC domain such as $\B ^d _\N$ is any function which respects the grading, direct sums, and the joint similarities which preserve its NC domain.)

Given any matrix $d-$tuple $X = \left( X_1 , \cdots , X _d \right) \in \C ^{(n\times n) \cdot d}$, the \emph{joint or outer spectral radius} of $X$ is: 
$$ \mr{spr} (X) := \lim \sqrt[\leftroot{-2}\uproot{5}2k]{\mr{Ad} ^{(k)} _{X, X^*} (I_n)}; \quad \quad \mr{Ad} _{X, X^*} (P) := X \left(I_d \otimes P \right) X^*. $$ 

\subsection{Domains and realizations of NC rational functions}

As described in the introduction, the domain of any NC rational function, $\mf{r}$, is 
$$ \dom{\mf{r}} := \bigsqcup _n \mr{Dom} _n (\mf{r} ); \quad \mr{Dom} _n (\mf{r} ) := \left\{ \left. X \in \C ^{(n\times n) \cdot d} \right| \ \fr (X) \ \mbox{is defined} \right\}. $$ 
NC rational functions which are regular at the origin, \emph{i.e.} $0 \in \dom{\mf{r}}$ have a powerful realization theory: There is a triple $(A,b,c) \in \C ^{(N\times N) \cdot d} \times \C ^N \times \C ^N$, called a \emph{realization of $\mf{r}$} so that for any $X \in \C ^{(n\times n) \cdot d}$, $$ \mf{r} (X) = b^* L_A  (X) ^{-1} c := b^* \otimes I_n L_A (X) ^{-1} c \otimes I_n, $$ where $L_A$ is the (monic) \emph{linear pencil} defined by:
$$ L_A (X) := I_N \otimes I_n - \sum _{j=1} ^d A_j \otimes X_j. $$ Moreover, $A$ can be assumed to be \emph{minimal} in the sense that $N$ is minimal, and this implies that the realization is both \emph{observable}:
$$ \bigvee b^* A ^\om = \C ^{1\times N},$$  and \emph{controllable} 
$$ \bigvee A^\om c = \C ^N, $$ see \emph{e.g.} \cite[Subsection 3.1.2]{HMS-realize}. The domains of NC rational functions which are regular at $0$ have a convenient description:
\begin{thm*}{ (\cite[Theorem 3.5]{Volcic}, \cite[Theorem 3.1]{KVV-ncrat})}
If $\mf{r}$ is an NC rational function regular at $0$ with minimal realization $(A,b,c)$, then,
$$ \dom{\mf{r}} := \bigsqcup _{n \in \N} \left\{ \left. X \in \C ^{(n\times n) \cdot d} \right| \ \mr{det} L_A (X) \neq 0 \right\}.$$
\end{thm*}

\subsection{A conjugation on Fock space}

It will be useful to consider the conjugation with respect to the standard orthonormal basis, $\{ L^\alpha 1 | \ \alpha \in \F ^d \}$, on Fock space, $C :\bH ^2 _d \rightarrow \bH ^2 _d$, defined by: $Cf = \ov{f}$, where  $$ f = \sum _{\om \in \F ^d} f_\om L^\om 1 \ \stackrel{C}{\longmapsto} \ov{f} := \sum _\om \ov{f} _\om  L^\om 1. $$ 
\begin{lemma}
The map $C : \bH ^2 _d \rightarrow \bH ^2 _d$ is a conjugation, \emph{i.e.} it is an anti-linear unitary involution. 
\end{lemma}
That is, the anti-linear operator $C$ is bijective, isometric, and $C ^2 = I$.
\begin{lemma}
The conjugation, $C: \bH ^2 _d \rightarrow \bH ^2 _d$ commutes with the left and right free shifts. If $F(L) \in \bH ^\infty _d$ then $C F(L) C =: \ov{F} (L) \in \bH ^\infty _d$ has the same operator norm and if 
$$ F(L) = \sum _\om F_\om L^\om \quad \mbox{then} \quad \ov{F} (L) = \sum _\om \ov{F} _\om L^\om. $$ 
Moreover for any $F,G \in \bH ^\infty _d$, and $f \in \bH ^2 _d$, $C F(L) G(L) C = \ov{F} (L) \ov{G} (L)$ and $C F(L) f = \ov{F} (L) \ov{f}$. 
\end{lemma}
Here, recall that any $F(L) \in \bH ^\infty _d$ can be identified with the power series:
$$ F(L) = \sum _{\om \in \F ^d} \hat{F} _\om L^\om; \quad \quad \hat{F} _\om := \ip{L^\om 1}{F(L)1}_{\bH ^2}, $$ in the sense that the Ces\`aro partial sums of this series converge to $F(L)$ in the strong operator topology \cite{DP-inv}.  
\begin{proof}
We have that 
\ba \| F(L) \|  & = &  \sup _{\| f \| _{\bH ^2 }  =1} \| F(L) f \|  \nn \\
& = & \sup \| C F(L) f \| \nn \\
& = & \sup \| \ov{F} (L) \ov{f} \| \nn \\
& = & \sup \| \ov{F} (L) f \|  = \| \ov{F} (L) \|. \nn \ea All of the other properties are easily verified.
\end{proof}
\begin{lemma}
If $f = \Theta (L) F$ is the inner-outer factorization of $f \in \bH ^2 _d$, then $\ov{f} = Cf$ has inner-outer factorization 
$\ov{f} = \ov{\Theta } (L) \ov{F}$. That is, $\Theta (L)$ is inner if and only if $\ov{\Theta} (L)$ is inner and $F \in \bH ^2 _d$ is outer if and only if $\ov{F}$ is outer. 
\end{lemma}
\begin{proof}
Clearly $\ov{\Theta } (L)$ is an isometry since $\ov{\Theta} (L) = C \Theta (L) C $, and $C$ is an anti-linear unitary. Suppose that $\ov{F}$ is not outer. Then there is a $g \in \bH ^2 _d$ so that $g \perp \bigvee R^\alpha F$. (Here recall $R_k := M^R _{\mathfrak{z} _k}$ are the isometric right free shifts.) Then,
\ba 0 & = & \ip{R^\alpha \ov{F} }{g}_{\bH^2} \nn \\
& = & \ip{Cg}{CR^\alpha \ov{F} } _{\bH^2} \quad \quad \mbox{($C$ is a conjugation)} \nn \\
& = & \ip{\ov{g}}{R^\alpha F} _{\bH^2}. \nn \ea This proves that $\ov{g} \perp \bigvee R^\alpha F$ so that $F$ is not outer. 
\end{proof}
For any $n \in \N$, we also define a conjugation, $\cc : \C ^n \rightarrow \C ^n$, with respect to the standard basis, $\{ e_j \} _{j=1} ^n$, 
$$ \cc \sum a_j e_j := \sum \ov{a_j} e_j. $$ For $c \in \C ^n$ we write $\ov{c} := \cc c$ and for $X \in \C ^{n\times n}$ we write $\ov{X} := \cc X \cc$, whose matrix in the standard basis is obtained by entry-wise complex conjugation.

\begin{lemma} \label{Szegocon}
If $Z \in \B ^d _n$ is a strict row contraction, $\ov{Z} = \cc Z \cc$ is also a strict row contraction with $\| \ov{Z} \| = \| Z \|$ and $\mr{spr} (\ov{Z}) = \mr{spr}(Z)$.
\end{lemma}

\section{Isomorphy of NC Rational functions and kernels in Fock space}

We consider NC rational functions $\fr(\mf{z})$ in $d$ noncommuting variables $\mf{z}_1, \cdots , \mf{z}_d$. Suppose that $\fr$ is defined in the row ball $\B ^d _\N$; since this domain contains the scalar point $0$ it follows that $\fr$ has a minimal realization $(A,b,c)$ of size $N$. Since $\B ^d _\N \subseteq \dom{\mf{r}}$, $L_A(Z)$ is invertible for all $Z$ in the row-ball by \cite[Theorem 3.5]{Volcic}. Since $\fr (Z)$ is a locally bounded (analytic) NC function in $\B ^d _\N$, it has a Taylor-Taylor series at $0$ with non-zero radius of convergence \cite[Chapter VII]{KVV}:
\ba  \fr (Z) & =  & b^* L_A(Z) ^{-1} c \nn \\
& = & \sum _{\om \in \F ^d } b^* A^\om c Z^\om \nn \\
& = & \sum \ipcN{b}{A^\om c} Z^\om. \label{gsum} \ea 

If we further assume that $\fr (Z)$ belongs to the Fock space then the above power series coefficients are square summable, 
$$ \sum _{\om \in \F ^d} \left| \ipcN{b}{A^\om c} \right| ^2 < +\infty, $$ and by \cite[Theorem 1.1]{Pop-freeholo} the above power series of Equation \ref{gsum} converges absolutely in $\B ^d _{\aleph _0}$ and uniformly (in operator-norm) on NC balls of radius $0 \leq r<1$. Assuming from now on that $\fr \in \bH ^2 _d$, we have that by \cite[Theorem 3.5, Theorem 3.10]{Volcic}, that 
$$ \fr _{y,v} (\mf{z} ):= y^* L_A(Z) ^{-1} v \in \mr{Hol} (\B ^d _\N ), $$ 
for any choice of $y,v \in \C ^N$ (since $L_A(Z)$ is invertible in $\B ^d _\N$). 

\subsection{NC Szeg\"{o} kernels are NC rational functions}

\begin{prop} \label{NCszego}
For any $Z \in \B ^d _n$ and $y, v \in \C ^n$, the NC Szeg\"o kernel vector is given by the formula: 
$$ K\{ Z , y , v \} = \sum _{\alpha \in \F ^d} \ipcn{Z^\alpha v}{y} L^\alpha 1. $$
This power series has radius of convergence $$ R \geq \frac{1}{\mr{spr} (Z) } \geq \frac{1}{\| Z \| } > 1. $$
In particular, the partial sums of the series converge uniformly in operator norm on every row ball of radius $r<R$, so that $K\{ Z , y , v \}$ belongs to the NC disk algebra $\A$.

 The image of any NC Szeg\"{o} kernel under the conjugation, $C$, is:
$$ C K \{Z , y , v \}  =  K \{ \cc Z \cc , \cc y , \cc v \} = K \{ \ov{Z} , \ov{y} , \ov{v} \}. $$ Any NC Szeg\"o kernel, $K\{ Z , y ,v \}$ is an NC rational function with (not necessarily minimal) realization $\left( \ov{Z} , \ov{y} , \ov{v} \right)$.  
\end{prop}
\begin{proof}
Given any strict row contraction $Z = \left( Z_1 , \cdots , Z_d \right) \in \B ^d _{\aleph _0}$, let us first calculate the radius of convergence of the power series formula for $K \{ Z , y , v \}$, where if $Z \in \B ^d _n$, then $y,v \in \C ^n$ (and we allow $n = \infty$). By \cite[Theorem 1.1]{Pop-freeholo}, the radius of convergence, $R$, of the NC power series
$$ K\{ Z ,y, v \} (W) =  \sum _\om \ipcn{Z^\om v}{y} W^\om, $$ 
is given by the Hadamard formula:
\ba \frac{1}{R} & = & \limsup _{k \rightarrow \infty} \left( \sum _{|\om | = k} \left| \ipcn{Z^\om v}{y} \right| ^2 \right) ^{\frac{1}{2k}} \nn \\
& = & \limsup \left( \sum _{|\om | = k}  \ipcn{y}{Z^\om v v^* (Z^\om ) ^* y}  \right) ^{\frac{1}{2k}} \nn \\
& = & \limsup \left( \ipcn{y}{\mr{Ad} ^{(k)} _{Z,Z^*} (vv^*) y}  \right) ^{\frac{1}{2k}} \nn \\
& \leq & \limsup \left( \| v \| ^2 _{\C ^n} \| y \| ^2 _{\C ^n} \| \mr{Ad} _{Z, Z^*} ^{(k)} (I_n) \|  \right) ^{\frac{1}{2k}} \nn \\
& = & \mr{spr} (Z) \leq \| Z \|  <1. \nn \ea This proves that 
$$ R \geq \frac{1}{\mr{spr} (Z)} \geq \frac{1}{\| Z \|} > 1, $$ so that $K \{ Z, y, v \} \in \mr{Hol} (R \B ^d _\N )$, and the Taylor-Taylor series of $K\{ Z , y, v \}$ at $0 \in \B ^d _1$ converges absolutely and uniformly on any NC row-ball $r \B ^d _\N$ of radius $0< r < R$. This proves, in particular, that $K\{ Z, y , v \} \in \mc{A} _d$. Since $\mc{A} _d \subsetneq \bH ^\infty _d \subsetneq \bH ^2 _d$, $K \{ Z, y ,v \}$ belongs to the Fock space. This can also be checked directly:
\ba \| K \{ Z , y ,v \} \| ^2 _{\bH ^2 _d} & = & \sum _\om \ipcn{y}{Z^\om v} \ipcn{Z^\om v}{y} \nn \\
& = & \sum _{m=0} ^\infty \sum _{|\om | =m} \ipcn{y}{Z^\om v v^* (Z^\om) ^* y} \nn \\
& = & \sum _{m=0} ^\infty \ipcn{y}{\mr{Ad} _{Z,Z^*} ^{(m)} (vv^* ) y} \nn \\
& \leq & \| y \| ^2 _{\C ^n } \| v \| ^2 _{\C ^n} \sum _{m=0} ^\infty \| Z \| ^{2m} \nn \\
& = & \frac{\| y \| ^2 \| v \| ^2}{1 - \| Z \| ^2} < \infty. \nn \ea

We now verify the reproducing formula. Let $\hat{K} \{ Z , y , v \} :=\sum  \ipcn{Z^\alpha v}{y} L^\alpha 1$, and compute:
\ba \ip{\hat{K} \{ Z , y ,v \} }{f}_{\bH ^2 _d} & = & \sum _{\alpha , \beta} \ipcn{y}{Z^\alpha v} \hat{f} _\beta \ip{L^\alpha 1}{L^\beta 1} _{\bH ^2 _d} \nn \\
& = & \sum _\alpha \ipcn{y}{Z^\alpha v} \hat{f} _\alpha \nn \\
& = & \ipcn{y}{\sum \hat{f} _\alpha Z^\alpha v} \nn \\
& = & \ipcn{y}{f(Z) v}. \nn \ea

Verifying the final assertion amounts to recognizing the geometric sum formula: If $K\{Z , y, v \}$ with $Z \in \B ^d _n$ is an NC Szeg\"o kernel, and $W \in \B ^d _m$, then
\ba K \{ Z , y,  v \} (W ) & = & \sum _\alpha \ipcn{Z^\alpha v}{y} W^\alpha \nn \\
& = & \sum \ipcn{\mf{C}y}{\mf{C}Z^\alpha v}W^\alpha \nn \\
& = & \sum \ipcn{\ov{y}}{\ov{Z} ^\alpha \ov{v}} W^\alpha \nn \\
& = & \sum (\ov{y} ^* \otimes I_m ) \ov{Z} ^\alpha \otimes W^\alpha (\ov{v} \otimes I_m ) \nn \\
& = & (\ov{y} ^* \otimes I_m ) \sum _{k=0} ^\infty  \left[ \left( \ov{Z} _1 \otimes I_m , \cdots , \ov{Z} _d \otimes I_m \right)  \bpm  I_n \otimes W_1 \\ \vdots \\ I_n \otimes W_d \epm \right] ^k (\ov{v} \otimes I_m )  \nn \\
& = & \ov{y} ^* L_{\ov{Z}} (W) ^{-1} \ov{v}. \nn \ea 
\end{proof}

\subsection{NC rational functions in Fock space are NC kernels}

If $\fr \in \bH ^2 _d$ is an NC rational function with minimal realization $(A,b,c)$ of size $N$, then $\ofr := C \fr \in \bH ^2 _d$ is also an NC rational function with the same norm and minimal realization $(\ov{A} , \ov{b} , \ov{c} )$. In particular, expanding $\ofr$ in a Taylor-Taylor series at $0 \in \B ^d _1$ yields:
\ba \ofr (Z) & = & \ov{b} ^* L_{\ov{A}} (Z) ^{-1} \ov{c} \nn \\
& = &  \sum _\om \ipcN{A^\om c}{b} Z^\om. \nn \ea 
Observe that $\ofr$ formally resembles an NC Szeg\"o kernel: $ \ofr \sim \wt{K} \{ A , b,  c \}. $ (Here, the tilde denotes that we do not yet know if the formal power series for the NC kernel at $\{ A, b ,c \}$ converges in $\B ^d _\N$, or if it belongs to the Fock space.)  It further follows that the original rational function, $\fr$, resembles the formal NC kernel $\wt{K} \{ \ov{A} , \ov{y} , \ov{v} \}$.

\begin{lemma} \label{evalA} 
Let $\fr$ be a rational function in $\bH ^2_d$ with minimal realization $(A, b, c)$ of size $N$. Then the unital homomorphism from NC polynomials into $\C ^{N \times N}$ defined by $p \mapsto p(A)$ is continuous in the $\bH ^2_d-$norm.  If $y, v$ are any vectors in $\mathbb C^N$ then $\fr _{y,v} (Z) := y^* (I - Z A ) ^{-1} v$ also belongs to $\bH ^2_d$. 
\end{lemma}
It follows that the evaluation $p\mapsto p(A)$ has a unique continuous extension to $\bH ^2_d$ which we write $f\mapsto f(A)$. In particular if $f_n$ are NC polynomials and $f_n\rightarrow f$ in the $\bH ^2_d$ norm, then $f_n(A)\to f(A)$ and for all NC polynomials $p,q \in \fp$ we have $(pfq)(A)=p(A)f(A)q(A)$.
\begin{proof}
We first observe that for any NC  polynomial $p(\fz )=\sum p_\om \fz^\om$ we have
$$  \ip{\ofr}{p}_{\bH ^2_d}  =  \sum _{\om \in \F ^d} \ipcN{b}{A^\om c} p _\om =\ipcN{b}{p(A)c}, $$ so that $\ofr$ also acts `like' the formal NC kernel vector $\wt{K} \{A, b, c \}$. Hence by Cauchy-Schwarz
$$ \left| \ipcN{b}{p(A)c} \right| \leq \|p \|_2 \|\ofr \|_2 = \| p \| _2 \| \fr \| _2. $$
Next, let $\{e_j \}$ be the standard basis for $\C ^N$. By the observability and controllability of $(A, b, c)$, there exist systems of NC polynomials $\{\beta_j\}$ and $\{\gamma_j\}$ such that
$$ c^*\gamma_i(A) =e_i^* \quad \text{and}\quad \beta_j(A)b =e_j; \quad 1 \leq j \leq N. $$
Then for any NC polynomial $p\in \fp$,
\ba e_i^*p(A)e_j & = & \ipcN{\gamma_i(A) ^* c}{p(A)\beta_j(A)b} \nn \\
& = & \ip{\ofr}{ \gamma _i p \beta _j}_{\bH ^2 _d}, \nn \ea
  and thus
\ba \left| e_i^*p(A)e_j \right| & \leq &  \|\gamma_i p \beta_j\|_2 \| \fr \|_2 \nn \\
& \leq & \|\ga _i (L) \|  \| \beta _j (R) \|  \|p\|_2\| \fr \|_2, \nn \ea 
so that for all $i, j$, the matrix entry evaluations $p\mapsto p(A)_{ij}$ are continuous for the $\bH ^2_d-$norm. By \cite[Theorem 3.10]{Volcic}, for any $1 \leq i, j \leq N$, the rational function $\ofr _{ij}(\fz )=e_i^*(I-\ov{A}\fz )^{-1}e_i$ is defined in the row ball and the norm estimate above shows that $\ofr _{ij}$ and hence $\fr _{ij}$ belongs to $\bH ^2_d$, since the formal inner product $p(A)_{ij}=\ip{ \ofr _{ij}}{p} _2$ defines a bounded linear functional on $\bH ^2_d$. Finally by taking linear combinations we have $\fr _{y,v} (\fz ) = y^* (I - \fz A ) ^{-1} v  \in \bH ^2 _d$ for all $y,v \in \C ^N$.  
\end{proof}

Let $A  \in \C ^{(n\times n) \cdot d }$ be a $d-$tuple of $n\times n$ matrices. Following \cite{SSS} we say $A$ is \emph{reducible} if it has a non-trivial jointly invariant subspace. If $A$ is not reducible we say it is \emph{irreducible}.
 
\begin{thm} \label{NCreg}
Let $\fr \in \bH ^2 _d$ be an NC rational function in $\bH ^2 _d$ with minimal realization $(A,b,c)$ of size $N$. Then $A$ has joint spectral radius $\mr{spr} (A) <1$ and is jointly similar to a point $W \in \B ^d _N$. If $A$ is irreducible, then one can choose $\| W \| = \mr{spr} (A) < 1$. Moreover, there are vectors $x,u \in \C ^N$ so that $\fr = K \{ W , x , u \}$ is analytic in an NC row-ball of radius $R \geq \| W \|  ^{-1} >1$. 
\end{thm}
\begin{proof}
We have that for any $d-$tuple $Z \in \B ^d _n$, 
$$ \ofr _{y,v } (Z) = \sum _{\om} \ipcN{A^\om v}{y} Z^\om \in \bH ^2 _d, $$ for any choice of $y,v \in \C ^N$ by Lemma \ref{evalA}. Hence,
\ba \infty & > & \| \ofr _{y,v} \| ^2 _2 \nn \\
& = & \sum _\om   \ipcN{y}{A^\om v} \ipcN{A^\om v}{y} \nn \\
& = & \sum _\om \ipcN{y}{A^\om v v^* (A^\om ) ^* y}, \nn \ea for any $y,v \in \C ^N$. Taking $v = e_j$ where $\{ e_j \}$ is the standard basis of $\C ^N$ and summing over $j$ yields:
\ba \infty & > & \sum _{j=1} ^N \sum _\om \ipcN{y}{A^\om e_j e_j ^* (A^\om ) ^* y} \nn \\
& = & \sum _{n=0} ^\infty \sum _{|\om | = n } \ipcN{y}{A^\om I_N (A^\om ) ^* y} \nn \\
& = & \sum _{n=0} ^\infty \ipcN{y}{ \mr{Ad} _{A, A^*} ^{(n)} (I_N )  y}, \nn \ea
where we view $A  = \left( A_1 , \cdots , A_d \right)  \in \C ^{(N\times N) \cdot d}$ as a row $d-$tuple of $N\times N$ matrices. Since this sum is finite, the general term must converge to $0$, 
$$ \ipcN{y}{ \mr{Ad} _{A, A^*} ^{(n)} (I_N )  y} \stackrel{\longrightarrow}{\mbox{\tiny $n \uparrow \infty$ } } 0, $$ for any $y \in \C ^N$. This proves that the positive semi-definite matrices:
$$ \mr{Ad} _{A, A^*} ^{(n)} (I_N) \rightarrow 0, $$ so that 
$A = \left( A_1 , \cdots , A_d \right)$, is a \emph{pure} and finite-dimensional $d-$tuple. By the multi-variable Rota-Strang theorem, \cite[Theorem 3.8]{PopRota} (see also \cite[Proposition 2.3, Remark 2.6]{SSS}), $A$ is jointly similar to a strict row contraction $X \in \B ^d _N$. That is, there is an invertible $S \in \C ^{N \times N}$ so that 
$$ A_k = S X_k S ^{-1}, \quad \quad 1 \leq k \leq d. $$ Hence, for any $Z \in \B ^d _n$ 
\ba \ofr _{b,c} (Z) & = & \sum _{\om \in \F ^d} \ipcN{A^\om c}{b} Z^\om \nn \\
& = & \sum \ipcN{X^\om S^{-1} c }{S^* b} Z^\om \nn \\
& = & K \{ X , x, u \} (Z); \quad \quad x:= S^* b, \ u := S^{-1} c \nn \ea  
By \cite[Lemma 2.4]{SSS}, if $A$ is irreducible then, 
$$ \mr{spr} \left( A \right) =  \min \left\{ \left. \| S ^{-1} A  S \| \right| \ S \in \GL (n) \right\}, $$ so that we can choose $\| X \| = \spr{A} <1$. Either way, since $A$ is jointly similar to the strict row contraction, $X \in \B ^d _\N$, $\spr{A} <1$. Setting $W := \mf{C} X \mf{C}$, Lemma \ref{Szegocon} and Proposition \ref{NCszego}, imply that $\fr = C \ofr = C K \{ W , x , u \} = K \{ \ov{W} , \ov{x} , \ov{u} \}$ is also an NC Szeg\"{o} kernel vector whose Taylor-MacLaurin series has radius of convergence 
$$ R \geq \frac{1}{\| W \|} > 1. $$ In particular, $\fr$ is analytic in an NC row-ball of radius $R>1$.
\end{proof}
\begin{cor} \label{ratkernel}
An NC rational function $\fr$ belongs to $\bH ^2 _d$ if and only if $\fr = K \{Z , y ,v \}$ for some finite point $Z \in \B ^d _N$, $N < + \infty$, and $y,v \in \C ^N$.
\end{cor}
\begin{remark}
The multivariable Rota-Strang theorem \cite[Theorem 3.8]{PopRota} is proven in a general multi-variable non-commutative context. For alternative proofs of this theorem applied to the special case of $d-$tuples of matrices, see \cite[Section 2]{SSS} and \cite[Theorem 1.7]{Pascoe}.
\end{remark}

\section{Regularity of NC rational functions in Fock space}

In this section, we will study varieties and spectra of NC rational multipliers. It will be convenient to briefly recall the concept of vectorization of matrices and completely bounded maps on matrices. Let $A \in \C^{m\times m}, B \in \C ^{n\times n}$, then $A \otimes B$ is an $mn \times mn$ matrix, but it can also be identified with a linear map on $\C ^{m \times n}$. To see the correspondence, for $Z \in \C ^{n\times m}$, we write $\wvec{Z}$ for the column vector of size $m \cdot n$ obtained by stacking the columns of $Z$ one on top of the other (in order from left to right). That is, dividing $Z \in \C ^{n \times m}$ into $m$ columns, $\mbf{z} _k \in \C ^n$ (see for example \cite[Section 4.2]{HornJohnson})
$$ Z = \bpm  \left. \mbf{z} _1 \right| \cdots \left| \mbf{z} _m \right. \epm \ \mapsto \wvec{Z} = \bpm \mbf{z} _1 \\ \vdots \\ \mbf{z} _m \epm \in \C ^{mn}. $$ By \cite[Lemma 4.3.1]{HornJohnson},
\[
\left( A \otimes B \right) \wvec{Z} = \wvec{B Z A^T}.
\]
This \emph{vectorization map} $\mr{vec} : \C ^{m \times n} \rightarrow \C ^{mn}$, $\mr{vec} (A) := \wvec{A}$, is clearly linear and invertible, and for any linear map $\ell \in \scr{L} (\C ^{m \times n} )$, we define the \emph{matrization} of $\ell$, $\wvec{\ell} \in \C ^{mn \times mn}$ by
$$ \wvec{\ell} \wvec{Z} :=  \wvec{\ell (Z)}, \quad \mbox{\emph{i.e.}} \ \wvec{\ell} = \mr{vec} \circ \ell \circ \mr{vec} ^{-1}. $$ In particular, if $\ell$ is any (completely bounded) linear map on the operator space $\C ^{m \times n}$, 
\be \ell (X) = \sum _{j=1} ^d A_j X B_j; \quad \quad A_j \in \C ^{m\times m}, B_j \in \C ^{n \times n}, \ X \in \C ^{m \times n} \label{cb} \ee then 
$$\wvec{\ell} = \sum B_j^T\otimes A_j. $$ 
The map $\ell \mapsto \wvec{\ell}$ has many nice properties, see \emph{e.g.} \cite[Section 3]{LM-dil} and \cite{Pascoe}. In particular, if we have a linear pencil $L_A(Z) = I_m \otimes I_n - \sum_{j=1}^d  A_j \otimes Z_j$, with $A_j \in \C ^{m\times m}, \, Z_j \in \C ^{n\times n}$, then for any $d$-tuple $Z \in \C ^{(n\times n) \cdot d}$, $L_A (Z) \in \C ^{mn \times mn}$, and it is clear that if we define $\ell \in \scr{L} (\C ^{n \times m} )$ by
\[
\ell (X) := X - \sum_{j=1}^d Z_j X A_j^T; \quad \quad X \in \C ^{n\times m}, 
\]
then $L_A (Z) = \wvec{\ell }$.
%We will be most interested in the case when $m=n$, in which case $A \otimes B$ and $B \otimes A$ are unitarily equivalent via the canonical shuffle. Hence,
Since $A \otimes B$ and $B \otimes A$ are unitarily equivalent via the canonical shuffle,
\[
 L_A (Z) \sim \wvec{\ell _{A,Z}},
\]
where $\ell _{A,Z} \in \scr{L} (\C ^{m \times n} )$ is defined by
$$ \ell _{A,Z} (X) := X - \sum _{j=1} ^d A_j X Z_j ^T; \quad \quad X \in \C ^{m \times n}. $$ 
It is now immediate that $L_A (Z)$ is singular if and 
only if $\ell _{A,Z}$ is.

%Recall from \cite{KlepVolcic-loci}, that a pencil $L_A(z) =  I - \sum_{i=1}^d z_i A_i$ is called irreducible, if $\cA = (A_1,\ldots,A_d)$ generates a simple algebra. Since we are working over $\C$, the only central simple finite-dimensional algebra is $M_n$, thus this condition is equivalent to requiring that $\cA$ is an irreducible $d$-tuple. An NC rational function $r$ is called \textrm{irreducible}, if it has (a necessarily minimal) realization   $r(Z) = b^* L_A(z)^{-1} c$, with $L_A$ irreducible.

\begin{prop} \label{prop:rational_spectral_radius}
%Let $r$ be an irreducible NC rational function. If $\overline{\B_{\N}^d} \subset \dom{r}$, then $\spr(A) < 1$.
Let $\fr$ be an NC rational function with $0 \in \dom{\fr}$. Let $(A,b,c)$ be a minimal realization of $\fr$ of size $N$.  Assume that $\spr{A} > 0$. Then, there exists a point $Z \in \C ^{(N \times N) \cdot d}$, such that $\|Z\| = \frac{1}{\spr{A}}$ and $Z \notin \dom{\fr }$.
\end{prop}
In the above $\spr{A} >0$ implies that $A$ is not jointly nilpotent, and hence $\fr$ is not a free polynomial.
\begin{proof}
%Since $A$ is irreducible, there exists $S \in \GL_n(\C)$, such that $S^{-1} A S = \rho(A) Y$, where $Y Y^* = I_n$ (see \cite[Proposition 4.3]{SSS}). Set $\rho(A) = \rho$ and assume that $\rho \geq 1$. Since applying a similarity to $A$ yields another minimal realization, then we may assume that $A = \rho Y$. Note that since $Y$ is a coisometry, then so is $\overline{Y}$, its entrywise complex conjugate, i.e,
%\[
%I = \overline{I} = \overline{Y Y^*} = \overline{Y} \overline{Y^*} = \overline{Y} \overline{Y}^*.
%\]
%For $X \in \overline{\B_n^d}$, we have that $L_A(X) = I_n \otimes I_n - \sum_{i=1}^n X_i \otimes A_i$. Thus we can treat $L_A(X)$ as a linear operator on $M_n$. For $T \in M_n$, $L_A(X)(T) = T - \sum_{i=1}^d A_i T X_i^T$. Since $\rho \geq 1$, $\frac{1}{\spr} \overline{Y} \in \overline{\fB_d}$. Thus, 
%\[
%L_A(\frac{1}{\spr} \overline{Y})(I) = I - \sum_{i=1}^d Y_i Y_i^* = 0.
%\]
%We conclude that $L_A(\frac{1}{\spr}\overline{Y})$ is singular. By \cite[Theorem 3.10]{Volcic}, we conclude that $\frac{1}{\spr}\overline{Y}$ is not in the domain of $r$, which is a contradiction.

Let $\rho := \spr{A}$. One can apply arbitrary similarities to $A$ to produce a new realization for the same NC rational function, $(A , b ,c ) \mapsto (S ^{-1} A S , S^* b , S^{-1} c)$. Applying a (unitary) similarity, we may assume that $A$ is block upper-triangular, where the blocks, $A^{(j)}$, $1\leq j \leq k$, on the diagonal are irreducible $d$-tuples. By \cite[Lemma 3.2]{Selfmaps} we may apply a subsequent block diagonal similarity, so that
\[
A = \begin{pmatrix} \rho _1 Y ^{(1)} & \star & \star \\ 0 & \ddots & \star \\ 0 & 0 & \rho_k Y ^{(k)} \end{pmatrix},
\]
where for each $1 \leq j \leq k$, we have $Y ^{(j)} (Y ^{(j)}) ^* = I$ and $\rho _j := \spr{A^{(j)}}$.  That is, each $d-$tuple $Y^{(j)}$ is a row co-isometry. Since $\rho = \max\{\rho_1, \ldots, \rho_k\}$, we may apply another similarity and assume that $\rho = \rho_1$. Let $Y := Y ^{(1)}$ and set
\[
Z := \begin{pmatrix} \frac{1}{\rho} \overline{Y} & 0 \\ 0 & 0 \end{pmatrix} \in \C ^{(N\times N)\cdot d}, \] be of the same size as $A$.
Note that since $Y \in \C ^{ (m \times m ) \cdot d }$, for some $m \leq N$ is a row co-isometry, then so is $\overline{Y} = \cc Y \cc$. In particular, $\|Z\| = \frac{1}{\rho}$. Consider the linear map $\ell [ \cdot ] := \ell _{A,Z} [ \cdot ] : \C ^{N\times N} \rightarrow \C ^{N\times N}$ defined as in the preceding discussion, $$ \ell_{A,Z} [X] = X - \sum_{j=1} A_j X Z_j ^T; \quad \quad X \in \C ^{N \times N}. $$ Let $P$ be the matrix with $I_m$ (the size of $Y ^{(1)}$) in the upper left corner and zeroes everywhere else, then
\ba \ell_{A,Z} [P] & = & P - \sum_{j=1}^d A_j P Z ^T  \nn \\
& = & \bpm I_m &  \\ & 0 _{N-m} \epm  -  \sum _{j=1} ^d \bpm \rho Y _j & \\ & 0 _{N-m} \epm  \bpm I_m & \\ & 0_{N-m} \epm \bpm   \frac{1}{\rho} (\ov{Y}  _j) ^T  & \\ & 0_{N-m} \epm  \\
& = & \bpm I_m  - \sum  Y_j Y_j ^* & \\ & 0_{N-m}  \epm \equiv 0. \nn \ea By the previous discussion, since $\ell _{A,Z}$ is singular, so is $L_A (Z)$. Since $0 \in \dom{\fr }$, by \cite[Theorem 3.10]{Volcic}, $\dom{\fr }$ coincides with the complement of the singularity locus of the pencil in its minimal realization. Hence $Z \notin \dom{\fr }$, as desired.
 \end{proof}
 
%\begin{cor} \label{cor:extended_analyticity_irr}
%Let $r$ be an irreducible NC rational function. If $\overline{\B_{\N}^d} \subset \dom{r}$, then $r \in H^{\infty}(\B_{\N}^d)$ and is analytic on a bigger ball.
% \end{cor}
% \begin{proof}
% By Lemma \ref{lem:rational_spectral_radius_irr} we have that the $d$-tuple $A_1,\ldots,A_d$ is similar to a strict row contraction. Thus, applying the similarity, we may assume that $A$ is a strict row contraction of norm $\rho = \rho(A)$. Therefore, 
%\[
%r(z) = b^* L_A(z)^{-1} c = \sum_{n=0}^{\infty} \sum_{|\alpha| = n} \langle A^{\alpha} c, b\rangle z^{\alpha}.
%\]
%In particular, $r$ is a kernel function and thus $r$ is bounded.
% \end{proof}
 
%Let us recall a few more notions from \cite{KlepVolcic-loci}. A domain of a rational function is called co-irreducible, if it is the complement of the determinantal zero locus of an irreducible pencil. By \cite[Proposition 4.4]{KlepVolcic-loci}, for every co irreducible domain $D \subset \C^d_{NC}$, there exist a unique $n \in \N$ and an irreducible pencil $L_A$, of size $n$, such that the set of all nc rational irreducible functions with domain $D$ is:
%\[
%\cR(D) = \left\{ b^* L_A^{-1}(z) c \mid b,c \in \C^n \setminus \{0\} \right\}.
%\] 
%Lastly, we recall \cite[Theorem 4.6]{KlepVolcic-loci}. Given an nc rational function $r$ with $0 \in \dom{r}$, $\dom{r} = D_1 \cap \cdots \cap D_k$, for some co-irreducible domains $D_1,\ldots, D_k$. Moreover, $r$ is a polynomial in $\{z_1,\ldots,z_d\} \cup \cR(D_1) \cup \cdots \cup \cR(D_k)$.

\begin{cor} \label{cor:extended_analyticity}
Let $\fr$ be an NC rational function, such that $\overline{\B_{\N}^d} \subseteq \dom{\fr }$. Then $\fr $ is bounded on $\B ^d _\N$ and analytic in $r \B ^d _\N$ for some $r>1$.
\end{cor}
\begin{proof}
Let $(A,b,c)$ be a minimal realization of $\fr$. If $\spr{A} = 0$, then $A$ is jointly nilpotent, $\fr $ is a polynomial and we are done. Therefore, we may assume that $\mr{spr}(A) > 0$. By Proposition \ref{prop:rational_spectral_radius}, there exists $Z \notin \dom{\fr }$, such that $\|Z\| = \frac{1}{\mr{spr}(A)}$. However, $\overline{\B_{\N}^d} \subset \dom{\fr }$. Thus, $\mr{spr}(A) < 1$ and $\fr $ is an NC kernel in the Fock space. 

%By \cite[Theorem 4.6]{KlepVolcic-loci} $r$ is a polynomials of the variables and the irreducible nc rational functions in $\cR(D_1) \cup \cdots \cR(D_k)$, where $\dom{r} = D_1 \cap \cdots \cap D_k$, with $D_1,\ldots,D_K$ irreducible. By Corollary \ref{cor:extended_analyticity_irr}, we know that for every $1 \leq j \leq k$ and every $q \in \cR(D_j)$, we have that $\overline{\B_{\N}^d} \subset D_j$ and thus $q$ is bounded on $\fB_d$ and is analytic in a bigger ball. Since $r$ is a polynomial in bounded functions, it is bounded. Furthermore, since each $D_j$ contains a $\rho_j \B_{\N}^d$ for some $r_j > 1$, then $\rho \B_{\N}^d \subset \dom{r}$, where $\rho = \min\{\rho_j \mid 1 \leq j \leq k\}$.
\end{proof}

\begin{cor} \label{cor:char_bdd_rat}
An NC rational function $\fr$ belongs to $\bH_d^{\infty}$ if and only if there exists $r >1$, such that $r\B_{\N}^d \subset \dom{\fr }$.
\end{cor}
\begin{proof}
One direction is Corollary \ref{cor:extended_analyticity} and the other is the ``moreover'' statement in Theorem~\ref{NCreg}.  %Corollary \ref{cor:boundary-regularity}.
\end{proof}

%%%%%%%%%%%%%%%%%%%%%%%%%%%%%%%%%%%%%%%%%%%%%%%%%%%

\iffalse

%%%%%%%%%%%%%%%%%%%%%%%%%%%%%%%%%%%%%%%%%%%%%%%%%%%

In summary, combining the above with our previous results, we have completed the proof of the main theorem which we repeat below:

\setcounter{thmA}{0}

\begin{thmA} 
Let $\mf{r}$ be an NC rational function with minimal realization $(A,b,c)$ of size $N$. Then the following are equivalent.
\begin{itemize}

\item[(i)] $\mf{r} \in \bH ^2_d$.

\item[(ii)] $\mf{r} \in \bH ^\infty _d$.

\item[(iii)] $\mf{r} \in \cA_d := \mr{Alg} (I, L) ^{- \| \cdot \| }$, the \emph{NC disk algebra}.

\item[(iv)] $\overline{\B^d_{\N}} \subset \dom{\mf{r} }$

\item[(v)] There exists $r > 1$, such that $r \B^d_{\N} \subset \dom{\mf{r}}$.

\item[(vi)] The joint spectral radius, $\mr{spr} (A)$, of $A$ is $<1$.

\item[(vii)] $\mf{r} = K \{ Z, y , v \}$ is an NC Szeg\"o kernel vector for some $Z \in \B ^d _N$ and $y, v \in \C ^N$.
\end{itemize}
\end{thmA}

%%%%%%%%%%%%%%%%%%%%%%%%%%%%%%%%%%%%%%%%%%%%%%%%%%%%

\fi

%%%%%%%%%%%%%%%%%%%%%%%%%%%%%%%%%%%%%%%%%%%%%%%%%%%%

We may now combine the above results to prove Theorem~\ref{thmA}:
\begin{proof}[Proof of Theorem A]
The implications (i)$\implies$(ii), (i)$\implies$(iii), and (i)$\implies$(iv) are contained in Theorem~\ref{NCreg}. On the other hand (ii)$\implies$(i) is trivial and (iii)$\implies$(ii) follows from the Rota-Strang theorem. Thus (i), (ii), and (iii) are equivalent, and each implies (iv). Next, (iv)$\implies$(v) is trivial, and (v)$\implies$(iii) follows from Proposition~\ref{prop:rational_spectral_radius}, so (i) through (v) are equivalent. Finally, (ii)$\implies$(vi) is contained in Proposition~\ref{NCszego}, and (vi)$\implies$(vii) and (vii)$\implies$(i) are trivial. This completes the proof. 
\end{proof}

\begin{eg}\label{eg:polydisk-fail}  The implication (vii)$\implies$(v), which says that any rational function bounded in the row ball is continuous up to the boundary (and then in fact analytic across the boundary, by (iv)), seems special to the row ball and fails, for example, in the NC polydisk. The NC polydisk, for $d\geq 2$, is the NC domain which at level $n$ is the domain $\D_n^d$ of all $d$-tuples of strictly contractive matrices:
  \[
    \D_n^d=\{ \left. (Z_1, \dots, Z_d) \, \right| \|Z_j\|<1 \text{ for all } j=1, \dots, d\}.
  \]
  A counterexample may be constructed in two variables as follows: the NC function
  \[
    f(Z,W) = \frac12 \left[ (I+Z)(I-Z)^{-1} +(I+W)(I-W)^{-1}\right]
  \]
  has $\re{f(Z,W)} \geq 0$ for all $(Z,W)\in \D_n^2$, at all levels $n$. It follows that its inverse Cayley transform
  \[
    g(Z,W) = (f(Z,W)-I)(f(Z,W)+I)^{-1}
  \]
  is bounded by $1$ at all levels. But at level $n=1$, a quick calculation shows that $g$ is the rational inner function
  \[
    g(z,w) = \frac{z+w-2zw}{2-z-w}
  \]
  in the bidisk $\mathbb D^2$, which does not extend continously from $\mathbb D^2$ to the boundary point $(1,1)$.

  The matricial character of the rational functions is also essential. If we look at, say, level $2$ of the row ball in $2$ dimensions, identified with the domain in $\mathbb C^8$
  \[
    \Omega =\left\{ \left. \left(Z=\begin{pmatrix} z_1 & z_2\\ z_3 & z_4\end{pmatrix}, W=\begin{pmatrix} w_1 & w_2\\ w_3 & w_4\end{pmatrix} \right) \, \right| \, \|ZZ^*+WW^*\|<1\right\},
  \]
  there will exist rational functions of the $8$ complex variables $z_1, ... ,z_4, w_1, \dots, w_4$ which are bounded in $\Omega$, but do not extend continuously to the boundary. An example is
  \[
    g(z_1, z_2, z_3, z_4, w_1, w_2, w_3, w_4) = \frac{z_1+w_4-2z_1w_4}{2-z_1-w_4}
  \]
  which is bounded by $1$ in $\Omega$, but does not extend continuously from $\Omega$ to the boundary point $Z=\bsm 1 & 0 \\ 0 & 0\esm, W=\bsm 0 & 0  \\ 0 & 1\esm$.

\end{eg}

\subsection{Singularity loci and spectra of NC rational functions} \label{sec:spectrum}

Let $\fr$ be an NC rational function. We would like to understand the NC variety of $\fr$ and to determine when it intersects $\B^d_{\N}$. Here, recall that the \emph{singularity locus} or (left) \emph{NC variety} of any $f \in \bH ^2 _d$ is:

$$ \mr{Sing}  (\fr ) = \bigsqcup_{n \in \N \cup \{ \infty \}}  \mr{Sing} _n (f); \quad \quad 
\mr{Sing} _n (f) := \left\{ \left. (Z,y) \in \B^d_n \times \C^n \setminus \{0\} \right| y^* f(Z) = 0 \right\},  $$ see \cite[Definition 3.2]{NC-BSO}.
In particular, if $(Z,y) \in \mr{Sing} _n (f)$ for some $n< + \infty$, then $f(Z)$ is a singular matrix.

\begin{cor} \label{cor:singularity_locus}
Let $\fr \in \bH_d^{\infty}$ be an NC rational function with minimal realization of size $N$.
\begin{itemize}

\item[(i)] If $\det \fr (Z) \neq 0$, for every $Z \in \overline{\B^d_k}$, for $k \leq N+1$, then $\fr (L)$ is invertible.

\item[(ii)] If $\det \fr (Z) \neq 0$, for every $Z \in \B^d_k$, for $k \leq N+1$, then $\sing{\fr } = \emptyset$.

\end{itemize}
\end{cor}

%If $\mf{r} = p \in \fp$ has homogeneous degree $m$, then the NC variety of $p$ is completely determined by the finite levels up to size at most $ \sum _{j=1} ^m d^j = \# \{ \om \in \F ^d | \  |\om | \leq m \}$ \cite[Proposition 5.13]{NC-BSO}.

\begin{proof}
By \cite[Algorithm 4.3]{HMS-realize}, the minimal realization of $\fr ^{-1}$ is of size at most $N+1$. First, assume that $\det \fr (Z) \neq 0$ for every $Z \in \overline{\B^d_k}$ and $k \leq N+1$. We want to show that $\fr ^{-1} \in \bH_d^{\infty}$. For this it suffices to show that the joint spectral radius of the minimal realization of $\fr ^{-1}$ is strictly less than $1$. Let $(A,b,c)$ be a minimal realization of $\fr ^{-1}$ so that $A$ has size at most $N+1$. Assume that $\mr{spr}(A) \geq 1$. Then, by Proposition \ref{prop:rational_spectral_radius}, there exists a point $Z$ of norm $\frac{1}{\spr{A}} \leq 1$ and size at most $N+1$ so that $Z$ is not in the domain of $\fr ^{-1}$. However, by assumption, $\overline{\B^d_{N+1}} \subset \dom{\fr^{-1}}$ and this is a contradiction. This proves (i).

To prove (ii), we need only to consider the case $\mr{spr}(A) = 1$. Fix an arbitrary $0 < r < 1$ and note that $\fr ^{-1} _r (Z) := \fr (rZ) ^{-1}$, has minimal realization $(rA,b,c)$. In particular, $\mr{spr}(rA) = r \, \mr{spr}(A) < 1$, and thus $\fr _r (L) = \fr (rL)$ is invertible by (i). Since this is true for every $0 < r < 1$, we have that $\mr{Sing}  (\fr ) = \emptyset$.
\end{proof}	

Now we can prove that the spectrum of a bounded NC rational function is determined on finite levels. To be more precise we make the following definition:

\begin{defn}
Let $f \in \bH_d^{\infty}$, we define the \textrm{finite spectrum} of $f(L)$ to be
\[
\fs(f(L)) = \overline{\bigcup_{Z \in \B ^d _\N} \sigma(f(Z))}.
\]
%The \textrm{finite spectral radius} of $f(L)$ is then
%\[
%\frad(f(L)) = \sup\{ |\lambda| \mid \lambda \in \fs(f(L))\}.
%\]
\end{defn}

\begin{lemma}
Let $f \in \bH_d^{\infty}$, then 
\[
 \fs(f(L)) \subseteq \sigma(f(L)).
\]
\end{lemma}
\begin{proof}
Let $Z \in \B_n^d$ and let $\lambda \in \sigma(f(Z))$. Let $0 \neq y \in \C^n$, be an eigenvector of $f(Z)^*$ associated with $\overline{\lambda}$. For any $v \in \C^n$,  we take the kernel function $K\{Z,y,v\}$ and $f(L)^* K\{Z,y,v\} = K\{Z,f(Z)^* y,v\} = K\{Z,\overline{\lambda} y,v\} = \overline{\lambda} K\{Z,y,v\}$.  Hence $\lambda$ in the spectrum of $f(L)$. Since the spectrum is closed, the claim follows.
\end{proof}

\begin{cor} \label{cor:spectrum}
Let $\fr \in \bH_d^{\infty}$ be an NC rational function with minimal realization of size $N$. Then $\lambda \in \sigma(\fr (L))$ if and only if there exists $Z \in \overline{\partial \B^d_k}$, for some $k \leq N+2$, such that $\lambda \in \sigma(\fr (Z))$. In particular,
\[
\sigma_{\N }(\fr (L)) = \sigma(\fr (L)).
\]
\end{cor}
\begin{proof}
By Corollary \ref{cor:singularity_locus}, $0 \in \sigma(\fr (L))$ if and only if there exists $Z \in \overline{\partial \B^d_k}$, for some $k < N+1$, such that $\det \fr (Z) = 0$. Now $\lambda \notin \sigma (\fr (L) )$ if and only if $\fr (L) - \lambda I$ is invertible. By \cite[Algorithm 4.3]{HMS-realize}, the minimal realization of $(\fr  - \lambda)^{-1}$ is of size at most $N+2$ and a second application of the previous corollary yields the claim.
\end{proof}

Given a unital Banach algebra, $\scr{A}$, we can consider the map $a \mapsto \sigma(a)$ from $\scr{A}$ to $2 ^\C _0$, the set of all compact subsets of $\C$ equipped with the Hausdorff metric. It is well-known that this spectral map, $\sigma : a \mapsto \sigma (a)$ is upper semi-continuous but generally not continuous. In \cite{ConwayMorrel}, Conway and Morrel have characterized the points of continuity of the spectral map on $\scr{L} (\cH)$ for $\cH$ a separable Hilbert space. To apply this result in our setting, we need to recall a few notions concerning the decomposition of spectrum. In \cite[Theorem 1.7]{DP-inv}, Davidson and Pitts proved that for every $f \in \bH_d^{\infty}$, $\sigma(f(L)) = \sigma_e(f(L))$, where $\sigma_e$ denotes the essential spectrum. Moreover, by \cite[Theorem 1.7]{DP-inv} and \cite[Corollary 1.8]{DP-inv}, $\sigma(f(L))$ is connected and is not a singleton.

Recall that an operator $T \in \scr{L} (\cH)$ is \emph{semi-Fredholm} if $\ran{T}$ is closed and at least one of $\ker{T}$, $\ran{T} ^\perp$ is finite dimensional. Following Conway and Morrel, we denote
by $\SF \subset \scr{L} (\cH)$, the collection of all semi-Fredholm operators and we define the \emph{index} of any $T \in \SF$ as:
\[
\ind(T) = \mr{dim} \ker{T} - \mr{dim}\ker{T^*} \in \Z \cup \{\pm \infty\}.
\] For $T \in B(\cH)$ and $n \in \Z \cup \{\pm \infty\}$ we further define:
\begin{align*}
\sigma _n (T) & = \{\lambda \in \sigma(T) \mid \lambda I - T \in \SF \text{ and } \ind(\lambda I - T) = n\},\\
\sigma _\pm (T)&  = \bigcup_{n \neq 0}  \sigma _n(T).
\end{align*}
By \cite[Theorem 1.7]{DP-inv}, every multiplier is injective. Thus, for every $\lambda \in \C$, $\lambda I - f(L)$ is semi-Fredholm. If $d > 1$, then the index has two possible values. If $\lambda I - f(L)$ is outer, then $\ind(\lambda I - f(L)) = 0$. If $\lambda I - f(L)$ is not outer, then it admits an inner-outer decomposition $\lambda I - f(L) = \theta(L) g(L)$. In particular, $\ker{\theta(L)^*} \subset \ker{\overline{\lambda} I - f(L)^*}$. However, $\dim{\ker{\theta(L)^*}} \neq 0$ and thus $\ind(\lambda I - f(L)) < 0$.  We summarize this discussion in the following lemma.

\begin{lemma} \label{lem:spec_of_mult}
For every $f \in \bH_d^{\infty}$, $\sigma(f(L)) = \sigma _{\pm}(f(L)) \cup \sigma_0 (f(L))$ and 
\begin{align*}
\sigma_0 (f(L)) & = \{ \lambda \in \sigma(f(L)) \mid \lambda I - f(L) \text{ is outer}\},\\
\sigma_{\pm} (f(L)) & = \{\lambda \in \sigma(f(L)) \mid \lambda I - f(L) \text{ has an inner factor}\}.
\end{align*}
Furthermore, $\fs(f(L)) \subseteq \overline{\sigma_{\pm} (f(L))}$.
\end{lemma}
\begin{proof}
The first two claims follow from the discussion preceding the lemma, and it remains prove the final claim. For every $Z \in \B_{\N}^d$ and $\lambda \in \sigma(f(Z))$, since $\lambda I - f(Z)$ is singular, $\lambda I - f(L)$ is not outer.
\end{proof}

\begin{remark}
We note that for every $Z \in \B_{\N}^d$ and $\lambda \in \sigma(f(Z))$, $\lambda I - f(L)$ actually has a non-trivial singularity locus on finite levels. Thus, in particular, $\lambda I - f(L)$ has a non-trivial Blaschke factor. (See Definition \ref{ncBSOdef} below for the definition of an NC Blaschke inner function.)
\end{remark}

\setcounter{thmA}{1}

\begin{thmA} \label{thm:spec_cont}
Let $\fr \in \bH_d^{\infty}$ be an NC rational function. Then $\fr (L)$ is a point of continuity of the spectrum.
\end{thmA}
\begin{proof}
Since the spectrum is connected, by \cite[Theorem 3.1(c)]{ConwayMorrel}, the spectrum is continuous at $\fr (L)$ if and only if $\sigma(\fr (L)) = \overline{\sigma _{\pm}( \fr (L) )}$. By Lemma \ref{lem:spec_of_mult}, we know that $\fs(\fr (L)) \subseteq \overline{\sigma _{\pm}(\fr (L))}$. Therefore, from Corollary \ref{cor:spectrum}, we conclude that $\sigma(\fr (L)) = \overline{\sigma _{\pm}( \fr (L) )}$.
\end{proof}

\section{Inner-outer factorization of NC rational functions in Fock Space} \label{ratinout}

As established in \cite{NC-BSO}, any element, $h$, of Fock space has a unique NC Blaschke-Singular-Outer factorization, $h = B \cdot  S \cdot f$ where $B \in \bH ^\infty _d$ is an NC Blashcke inner function, $S \in \bH ^\infty _d$ is NC singular inner and $f \in \bH ^2 _d$ is NC outer. Namely, as proven by Popescu and Davidson-Pitts, $h = \Theta \cdot f$ has a unique NC inner-outer factorization, where $\Theta \in \bH ^\infty _d$ is inner, \emph{i.e.} an isometric left multiplier, and $f$ is NC outer, \emph{i.e.} cyclic for the right shifts. In \cite{NC-BSO} we proved that the inner factor, $\Theta$ can be further uniquely factored as the product of an NC Blaschke inner function, $B$, containing all the vanishing or zero information of $h$, and an NC singular inner function which is pointwise invertible in $\B ^d _\N$. 

\begin{defn}{ (\cite[Definition 3.6]{NC-BSO})} \label{ncBSOdef}
An NC inner $\Theta \in \bH ^\infty _d$ is:
\bn
    \item[(i)] \emph{Blaschke} if $\displaystyle \ran{\Theta (L)} = \left\{ h \in \bH ^2 _d \left|  \ y^* h(Z) \equiv 0 \ \forall (Z,y) \in \mr{Sing} (\Theta )  \right. \right\}. $ 
    \item[(ii)] \emph{singular} if $\Theta$ is pointwise invertible in $\B ^d _{\aleph _0}$. 
\en
\end{defn}

Recall, as discussed in the introduction, that a rational function, $\fr$ in the classical Hardy space, $H^2$, has a simple inner-outer factorization. Namely, $\fr = b \cdot f$, where $b$ is Blaschke inner (in fact a finite Blaschke product, hence rational) and $f$ is outer. In particular, $\fr$ has no singular inner factor. In this section we obtain analogues of these results in the NC setting of Fock space. 

\begin{thm} \label{NCratinout}
Let $\fr \in \bH ^2 _d$ be an NC rational function with minimal realization $(A, b, c)$ of size $N$. If $\fr = \Theta \cdot f$ is the NC inner-outer factorization of $\fr$ then both $\Theta$ and $f$ are NC rational functions and the inner factor is NC Blaschke. The NC outer factor, $f$, has a realization $(A , \wt{b} , c )$ for some $\wt{b} \in \C ^N$ and the minimal realization of $\Theta$ is of size at most $2N +1$.
\end{thm}
\begin{proof}
We have that $\fr (\zeta ) = b^* (I - \fz A ) ^{-1} c$ belongs to $\bH ^2 _d$ and $b,c \in \C ^N$ (where $N$ is minimal) if and only if $\fr = K \{ Z, y ,v \} $ for some $Z \in \B ^d _N$, and some $y,v \in \C ^N$ by Corollary \ref{ratkernel} and Theorem \ref{NCreg}. Moreover, there is a similarity, $S \in \C ^{N \times N}$ so that $\ov{Z} = S^{-1} A S$, $\ov{y} = S^* b$ and $\ov{v} = S^{-1} c$. Calculate that
\ba f & = & \Theta (L) ^* \Theta  (L) f \nn \\
& = &  \Theta (L) ^* \fr = \Theta (L) ^* K \{ Z , y , v \} \nn \\
& = & K \{ Z , \Theta (Z) ^* y , v \}, \nn \ea and it follows that there is some $\wt{b} \in \C ^N$ so that 
$$ f (\fz ) = \wt{b} ^* (I - \fz A ) ^{-1} c. $$ Hence $(A ,\wt{b} ,c )$ is a (not necessarily minimal) realization of $f$ of size $N$. Since $f$ is outer, it is invertible in the NC unit ball \cite[Lemma 3.2]{JM-freeSmirnov} \cite[Theorem 4.2]{NC-BSO}, and it follows that 
$$ \Theta (Z) = \fr (Z) f(Z) ^{-1}, $$ is also an NC rational function with minimal realization of size at most $2N+1$ \cite[Algorithm 4.3]{HMS-realize}. That is, $f(Z) ^{-1}$ has a realization of size at most $N+1$ and $\fr \cdot f ^{-1}$ then has a descriptor realization of size at most $N + N +1 = 2N +1$. 

Finally, by Theorem \ref{thmA}, since $\Theta  = \Theta (L) 1 \in \bH ^2 _d$ is an NC rational function, $\Theta \in \A$ belongs to the NC disk algebra. By \cite[Theorem 6.10]{NC-BSO} any NC inner function in $\A$ is necessarily Blaschke, and we conclude that $\Theta$ is an NC Blaschke inner function.
\end{proof}

\begin{cor} 
Any inner NC rational function is NC Blaschke.
\end{cor}

\begin{cor} \label{NCratouter}
An NC rational function, $\fr \in \bH ^2 _d$ with minimal realization $(A,b,c)$ of size $N$ is NC outer if and only if $\mr{Sing} _{N+1} (\fr ) = \emptyset$.
\end{cor}
\begin{proof}
If $\fr = \Theta \cdot f$ is the inner-outer factorization of $\fr$, then $\Theta$ is NC Blaschke by the previous corollary, so that $\Theta \neq I$ if and only if $\mr{Sing} ( \fr ) \neq \emptyset$. By Corollary \ref{cor:singularity_locus}, $\mr{Sing} (\fr ) \neq \emptyset$ if and only if $\mr{Sing} _{N+1} (\fr ) \neq \emptyset$. 
\end{proof}

\begin{cor}
An NC rational function $\mf{r} \in \bH ^2 _d$ is outer if and only if the radius of convergence of the Taylor-Taylor series of $\mf{r} ^{-1}$ at $0 \in \B ^d _1$ is at least $1$.
\end{cor}
\begin{proof}
By Theorem \ref{NCratinout}, the inner factor, $\Theta$ of $\mf{r}$ is NC Blaschke, so that $\mf{r}$ is outer if and only if $\mf{r}$ is pointwise invertible in $\B ^d _{\aleph _0}$ by \cite[Theorem 4.2]{NC-BSO}. In particular, $\mf{r}$ is NC outer if and only if $\mf{r} (rL) ^{-1} \in \bH ^\infty _d$ for any $0<r<1$, and this happens if and only if the radius of convergence of the Taylor series at $0$ is at least $1$.
\end{proof}

\begin{eg}
Let $V(L)$ be any NC rational inner and set $\fr (L) := a I + b V(L)$, $a,b \in \C$, and suppose $a \neq 0$. If $| b/a | \leq 1$ then $\fr (L)$ will be outer by \cite[Lemma 3.3]{JM-freeSmirnov}. Otherwise if $w := b/a $ is such that $|w| >1$, let $z=-1/w \in \D$. Then,
\ba \fr (L) & = & w (z-V(L)) \nn \\
& = & w \underbrace{\left( z - V(L) \right) \left(I - \ov{z} V(L) \right) ^{-1}}_{\mbox{NC rational inner, Blaschke}} \cdot \underbrace{\left( I - \ov{z} V(L) \right)}_{\mbox{NC outer}}. \nn \ea In the above, one can verify that the M\"obius transformation of an isometry is always an isometry so that the first factor is NC inner and rational, and therefore NC Blaschke. The second factor has the form $1-B$ for a contractive $B \in \bH ^\infty _d$ and is therefore NC outer by \cite[Lemma 3.3]{JM-freeSmirnov}.  
\end{eg}

\section{Inner-outer factorization of NC polynomials}  \label{ncpoly}

In this section we apply the results of the previous section to NC polynomials, and compute some examples. It is well-known that if
\[
  p(\mf{z}) = b^*(I-A\mf{z})^{-1}c
\]
is a minimal realization of an NC polynonmial $p$, then the matrices $A$ are jointly nilpotent; in particular, if $w$ is any word of length $|w|>\deg{p}$, then $A^w=0$. To see this, observe that if $|w|>\deg p$, then $c^*A^wb=0$, and therefore
\[
  c^*\gamma(A)A^w\beta(A)b=0
\]
for all NC polynomials $\beta$ and $\gamma$. Since the minimal realization is observable and controllable, we conclude that $d^*A^we=0$ for all vectors $d,e$. Let $p \in \fp$ be any NC polynomial with inner-outer factorization $p = \Theta \cdot f.$  By Corollary~\ref{ratkernel} and the above remarks, $p = K \{Z ,y , v \}$ for some jointly nilpotent $Z$ of order $\mr{deg} (p)$. 

\begin{thm} \label{NCration}
Let $p = \Theta \cdot f $ be the NC inner-outer factorization of $p=K\{ Z, y ,v \}$. Then $f =q \in \fp$ is an NC polynomial of degree $\mr{deg} (q) \leq \mr{deg} (p)$, and $\Theta (Z) = p(Z) q(Z) ^{-1}$ is an inner NC rational function. The inner factor of $p$ is NC Blaschke.
\end{thm}
\begin{proof}
The inner factor of $p \in \fp$ is NC Blaschke by Theorem \ref{NCratinout}. Observe that 
\ba f & = & f(L) 1 = \Theta (L) ^* p(L) 1 \nn \\
& = & \Theta (L) ^* K\{ Z, y , v \} \nn \\
& = & K\{ Z , \Theta (Z) ^* y , v \}. \nn \ea 
Since $Z$ is a jointly nilpotent row contraction of order $n $, $n := \mr{deg} (p)$ we have that
\ba f & = &  K \{ Z , \Theta (Z) ^* y ,v \} \nn \\
& = & \sum _{|\alpha | \leq \mr{deg} (p)} \left( Z^\alpha v , \Theta (Z) ^* y \right) _{\C ^{m_n}} L^\alpha 1, 
\nn \ea is also an NC polynomial, $f=q \in \fp$ of degree less than or equal to $\mr{deg} (p)$.
In particular, 
$$ \Theta (Z) = p(Z) q(Z) ^{-1}, $$ is an NC rational inner function.
\end{proof}
Given a free polynomial, $p \in \fp$, let $T_p$ be the hereditary subset of $\F ^d$ determined by the non-zero coefficients of $p$. That is, if $S_p := \{ \alpha \in \F ^d | \ p_\alpha \neq 0 \}$, then $T_p$ is the set of all $\beta \in \F ^d$ so that there exists a $\ga \in \F ^d$ such that $\ga \beta \in S_p$. Let $|T_p|$ be the number of elements in the tree, $T_p$. 
\begin{cor} \label{nonouterpoly}
An NC polynomial $p \in \fp$ is NC outer if and only if $\mr{Sing} _{|T_p|} (p)= \emptyset$. In particular, $p$ is NC outer if and only if $p(Z)$ is pointwise invertible in $\B ^d _\N$.
\end{cor}
\begin{proof}
A straightforward refinement of the argument of \cite[Proposition 5.13]{NC-BSO} shows that $\mr{Sing} _{|T_p|} (p)$ is empty if and only if the full NC variety $\mr{Sing} (p)$ is empty (this includes the infinite level). The claim now follows from Corollary \ref{NCratouter}.
\end{proof}

\begin{eg}
Let $p (L) = I + L_1 + L_1 L_2$. Then we know that the outer factor, $q$, of $p$ must be of the form:
$$ q(L) = a I + bL_1 +c L_2 + d L_1 L_2. $$ where we may assume $a>0$. This outer factor must obey $q(L) ^* q(L) = p(L) ^* p(L)$, and it must maximize $| q(0) | ^2$. This gives the equations: 
$$ p(L) ^* p(L) = 3I + \underbrace{L_1 +L_2 + L_1 L_2}_{=:g(L)} + g(L) ^*, $$ and 
%$$ q(L) ^* q(L) = (a^2 + b^2 + c^2 + d^2) I + \ov{a} b L_1 + (\ov{a} c + \ov{b} d ) L_2 + \ov{a} d L_1 L_2 + \mbox{conjugate terms}. $$
$$ q(L) ^* q(L) = (a^2 + b^2 + c^2 + d^2) I + a b L_1 + (a c + \ov{b} d ) L_2 + a d L_1 L_2 + \mbox{conjugate terms}. $$ 
Equating coefficients yields:
%$$ \ov{a} c + \frac{1}{|a| ^2} =1, \quad b = \frac{1}{\ov{a}} = d, $$ and 
%$$ 0 = r(t) := t^4 - 3 t^3 + 3 t^2 - 2t +1; \quad \quad t := |a| ^2. $$
$$ a c + \frac{1}{a ^2} =1, \quad b = \frac{1}{a} = d, $$ and 
$$ 0 = r(t) := t^4 - 3 t^3 + 3 t^2 - 2t +1; \quad \quad t := a ^2. $$ 
One can check $t=1$ is a root of this polynomial (taking $a=1$ gives $p(L)$). Factoring this root out gives the cubic polynomial 
$$t^3 - 2t^2 +t -1, $$ which has two complex roots and one real root $t_0 \simeq 1.7549 >1$. 
It follows that if we set $a= \sqrt{t_0}$, and set 
$$ b = d = \frac{1}{\sqrt{t_0}}, \quad c = \frac{1}{\sqrt{t_0}} \left( 1 - \frac{1}{t_0} \right),$$ then the $q(L)$ with these coefficients is the outer factor of $p$. 
\end{eg}

\begin{eg}
The Drury-Arveson space $H^2_d$ can be viewed as the subspace of the full Fock space spanned by the NC kernels $K_z := K\{ z , 1 ,1 \}$ for $z \in \B ^d = \B ^d _1 = \left( \C ^{1 \times d} \right) _1$. This subspace is left shift co-invariant and it follows that the multiplier algebra, $H^\infty _d$, of $H^2_d$ is a complete quotient of $\bH^{\infty}_d$. 

In forthcoming work, Aleman, Hartz, McCarthy and Richter have proven that any element $h \in H^2 _d$ has a unique \emph{quasi-inner--free outer factorization} \cite{AHMcR-freeouter}. That is $h = \theta \cdot f$ where $\theta$ is the compression of an inner left multiplier $\Theta \in \bH ^\infty _d$ to $H^\infty _d$, and $f = P_{H^2} F$ is the projection of an outer $F \in \bH ^2 _d$ onto $H^2 _d$. In particular, there are many examples of outer functions $h \in H^2 _d$ which are not free outer, in the sense that there is no NC outer $H \in \bH ^2 _d$ so that $P_{H^2} H = h$. (A function $h\in H^2_d$ is called {\em outer} if it is cyclic for the algebra of multipliers of $H^2_d$, which happens if and only if there is a sequence of polynomials $q_n(z_1, \dots, z_d)$ such that $h\cdot q_n\to 1$ in the Hilbert space norm.) A simple example is given by the following polynomial \cite{Richter}; we thank the authors of \cite{AHMcR-freeouter} for their permission to include it here:
$$ p(z_1 , z_2 ) = 1 - 2 z_1 z_2. $$ To see that this polynomial is outer in $H^2 _2$, we argue as follows: the norm of a polynomial $q(z_1, z_2)=\sum a_{mn}z_1^mz_2^n$  in $H^2_2$ is by definition
\begin{equation}\label{eqn:DA-norm-def}
  \|q\|_{H_2^2}^2 =\sum_{m,n} \frac{m! n!}{(m+n)!} |a_{mn}|^2.
\end{equation}
We consider the Dirichlet space, $\scr{D}$, of analytic functions $f=\sum a_n z^n$ in the unit disk $|z|<1$, equipped with the Hilbert space norm
\begin{equation}\label{dirichlet-norm-def}
  \|f\|_\scr{D} ^2 = \sum_{n=0}^\infty (n+1)|a_n|^2.
\end{equation}
From (\ref{eqn:DA-norm-def}) and Stirling's formula, there is an absolute constant $C$ such that for any one-variable polynomial $q(z)=\sum_n a_n z^n$, we have
\begin{equation}\label{eqn:DA-dirichlet-estimate}
  \|q(2z_1z_2)\|_{H_2^2}^2 \leq C\sum_{n=0}^\infty \sqrt{n+1}|a_n^2|\leq C\sum_{n=0}^\infty (n+1)|a_n|^2 =\|q(z)\|_\scr{D} ^2.
\end{equation}
From \cite[Lemma 8]{BrownShields-cyclic}, the function $(1-z)$ is cyclic for the Dirichlet space $\scr{D}$, hence there exists a sequence of one-variable polynomials $q_n(z)$ such that $(1-z)\cdot q_n(z) \to 1$ in the $\scr{D}-$ norm as $n\to \infty$. It then follows from (\ref{eqn:DA-dirichlet-estimate}) that $(1-2z_1z_2)\cdot q_n(2z_1z_2)\to 1$ in $H^2_2$, as desired.

%assume the converse. Then if $S = (S_1 , S_2)$ is the Arveson $2-$shift, $p(S) ^*$ is not injective so that there is an $h \in H^2 _d$ so that $0 = p(S) ^* h = p(L) ^* h$, or equivalently,
%$$ 2 L_2 ^* L_1 ^* h = h, $$ so that $h$ is an eigenvector of $p(L)^*$. By a result of Davidson-Pitts \cite[Theorem 2.6]{DP-inv}, $h = k_z \in H^2 _d$ is a Szeg\"o kernel for some point $z = (z_1, z_2) \in \B ^2$, and therefore $z_1 z_2  = \frac{1}{2}$. Setting $t := |z_1| ^2$, we obtain that 
%$$ |z_1 | ^2 + |z_2 | ^2 = \frac{1}{4t} + t =: f(t). $$ This function is concave up and has an absolute minimum on $[0,1]$ of $1$ at $t= 1/2$. This proves that $|z_1 | ^2 + |z_2 | ^2 \geq 1$, which contradicts that $z \in \B ^2$. 

A natural lift of the above polynomial is the symmetric polynomial $P(\mf{z}_1,\mf{z}_2) = 1 - \mf{z}_1 \mf{z}_2 - \mf{z}_2 \mf{z}_1$. As in \cite[Example 6.8]{NC-BSO}, we have that the inner-outer decomposition of $P$ is
$$
P(\mf{z}_1,\mf{z}_2 ) 
 =  \underbrace{\left( \frac{1}{\sqrt{2}} - V(\mf{z}) \right)\left( 1 - \frac{1}{\sqrt{2}} V(\mf{z} ) \right) ^{-1}}_{-\mu _{\frac{1}{\sqrt{2}}} (V) =: \Theta _P } \cdot \underbrace{\sqrt{2} \left( 1 - \frac{1}{\sqrt{2}} V(\mf{z} ) \right)}_{=:F_P}. $$
Here, $V(\mf{z}) = V (\mf{z} _1 , \mf{z} _2 ) := \frac{1}{\sqrt{2}} \left( \mf{z}_1 \mf{z}_2 + \mf{z}_2 \mf{z}_1 \right)$ is inner, and evaluation of the M\"obius transformation, $\mu _{-\frac{1}{\sqrt{2}}} (z) = (z + 1/ \sqrt{2} ) (1 + z/\sqrt{2}) ^{-1}$, at any isometry is again an isometry, so that $-\mu _{\frac{1}{\sqrt{2}}} (V) = \Theta _P$ is again NC inner (in fact NC rational and Blaschke). The second factor $F_P$ is, up to a constant factor, the identity plus a (strictly) contractive left multiplier, and therefore it is NC outer by \cite[Lemma 3.3]{JM-freeSmirnov}. Uniqueness of the quasi-inner--free outer factorization of $p$ \cite{AHMcR-freeouter}, now implies that $p$ cannot be free outer. If $p$ was free outer, then $p = 1 \cdot p$ would be the unique quasi inner--free outer factorization of $p$, but clearly $p = \theta _p \cdot f_p$ where $ \theta _p ( z_1 , z_2 ) = \Theta _P (z_1 ,z_2 )$ and $f_p (z_1 ,z_2) = F_p (z_1 ,z_2)$ is a second quasi-inner--free outer factorization of $p$, a contradiction.

Since $P \in \fp$ is a non-outer free polynomial of homogeneous degree $2$, Corollary \ref{nonouterpoly} above implies that there is a point $Z \in \B ^d _N$ for $N \leq 5 =|T_P|$ so that $\mr{det} P (Z) =0$. To construct a point in the singularity set of $P$, we write $V = V(L_1,L_2)$, which is an isometry on $\bH^2_d$. Observe that $F_P = ( I - \frac{1}{\sqrt{2}} V )  \in \bH ^\infty _d$ is a bounded, invertible left multiplier so that  
$$ f (\mf{z} ) := \sqrt{2} F_p (\mf{z} ) ^{-1} = \left( 1 - \frac{1}{2} ( \mf{z_1} \mf{z} _2 + \mf{z} _2 \mf{z} _1 ) \right) ^{-1} \in \bH ^\infty _d, $$ is also a bounded left multiplier. An easy geometric series argument now verifies that if $f = f(L) 1$, then $P(L)^* f =0$. Moreover, $f$ is an NC rational function and a Schur complement argument shows that a realization for $f$ is given by the formula:
$$ f (\mf{z} ) = \left( 1, 0 , 0 \right) \left( I_3 + \underbrace{\frac{1}{\sqrt{2}} \bpm 0 & 0 & 1 \\ 1 & 0 & 0 \\ 0 & 0 & 0\epm}_{=:-A_1} \mf{z} _1 + \underbrace{\frac{1}{\sqrt{2}} \bpm 0 & 1 & 0 \\ 0 & 0 & 0 \\ 1 & 0 & 0 \epm}_{=:-A_2} \mf{z}_2 \right)^{-1} \bpm 1 \\ 0 \\ 0 \epm. $$ That is, $f$ has the realization $(A , b, c )$ where $b = c = \bsm 1 \\ 0 \\ 0 \esm \in \C ^3$. This $A \in \C ^{(3\times 3)\cdot 2}$ is a row contraction but it is not a strict row contraction. However it is easy to check that $\mr{spr} (A) <1$, and that $A = D W D^{-1}$ is similar to a strict row contraction $W \in \B ^2 _3$ where 
$$ W _1 = - \bpm 0 & 0 & 2^{-3/4} \\ 2 ^{-1/4} & 0 & 0 \\ 0 & 0 & 0 \epm, \quad W_2 = - \bpm
0 & 2 ^{-3/4} & 0 \\ 0 & 0 & 0 \\ 2 ^{-1/4} & 0 & 0 \epm, \quad \mbox{and} \quad D = \bpm 1 & & \\  & 2 ^{1/4} & \\ & & 2 ^{1/4} \epm. $$ It then follows that $f = K \{ W , e_1 , e_1 \}$ where $e_1 = b = c \in \C ^3$ is the first standard basis vector of $\C ^3$. Since $P(L) ^* f =0$, we have that 
$$ 0 = P(L) ^* f = P(L) ^* K\{ W, e_1 , e_1 \} = K \{ W , P(W) ^* e_1 , e_1 \}, $$ and therefore (since $e_1$ is cyclic for the unital algebra generated by $W$) $P(W) ^* e_1 =0$. Indeed, this is easily checked:
\ba P(W)^* & = & I_3 - W_2 ^* W_1 ^* - W_1 ^* W_2 ^* \nn \\
& = & \bpm 0 & 0 & 0 \\ 0 & - \frac{1}{2} & 0 \\ 0 & 0 & - \frac{1}{2} \epm. \nn \ea 
It follows that $(W , e_1 ) \in \mr{Sing} _3 ( P)$ is a finite dimensional point in the NC variety of $P$. 
\end{eg}

\bibliographystyle{abbrvnat}
\bibliography{Bibs/ncBla}

\begin{thebibliography}{25}
\providecommand{\natexlab}[1]{#1}
\providecommand{\url}[1]{\texttt{#1}}
\expandafter\ifx\csname urlstyle\endcsname\relax
  \providecommand{\doi}[1]{doi: #1}\else
  \providecommand{\doi}{doi: \begingroup \urlstyle{rm}\Url}\fi

\bibitem[Agler and McCarthy(2015)]{AM15}
J.~Agler and J.~E. McCarthy.
\newblock Global holomorphic functions in several noncommuting variables.
\newblock \emph{Canad. J. Math.}, 67\penalty0 (2):\penalty0 241--285, 2015.

\bibitem[Aleman et~al.(2020)Aleman, Hartz, McCarthy, and
  Richter]{AHMcR-freeouter}
A.~Aleman, M.~Hartz, J.~McCarthy, and S.~Richter.
\newblock Free outer functions in complete {P}ick spaces.
\newblock In preparation, 2020.

\bibitem[Ball et~al.(2016)Ball, Marx, and Vinnikov]{BMV}
J.~A. Ball, G.~Marx, and V.~Vinnikov.
\newblock Noncommutative reproducing kernel {H}ilbert spaces.
\newblock \emph{J. Funct. Anal.}, 271\penalty0 (7):\penalty0 1844--1920, 2016.

\bibitem[Brown and Shields(1984)]{BrownShields-cyclic}
L.~Brown and A.~L. Shields.
\newblock Cyclic vectors in the {D}irichlet space.
\newblock \emph{Trans. Amer. Math. Soc.}, 285:\penalty0 269--303, 1984.

\bibitem[Conway and Morrel(1979)]{ConwayMorrel}
J.~B. Conway and B.~B. Morrel.
\newblock Operators that are points of spectral continuity.
\newblock \emph{Integral Equations Operator Theory}, 2\penalty0 (2):\penalty0
  174--198, 1979.

\bibitem[Davidson and Pitts(1999)]{DP-inv}
K.~R. Davidson and D.~R. Pitts.
\newblock Invariant subspaces and hyper-reflexivity for free semigroup
  algebras.
\newblock \emph{Proceedings of the London Mathematical Society}, 78\penalty0
  (2):\penalty0 401--430, 1999.

\bibitem[Davidson et~al.(2017)Davidson, Dor-On, Shalit, and Solel]{DDSS}
K.~R. Davidson, A.~Dor-On, O.~M. Shalit, and B.~Solel.
\newblock Dilations, inclusions of matrix convex sets, and completely positive
  maps.
\newblock \emph{Int. Math. Res. Not. IMRN}, \penalty0 (13):\penalty0
  4069--4130, 2017.
\newblock ISSN 1073-7928.
\newblock \doi{10.1093/imrn/rnw140}.
\newblock URL \url{https://doi.org/10.1093/imrn/rnw140}.

\bibitem[Helton et~al.(2013)Helton, Klep, and McCullough]{HKM-relax}
J.~W. Helton, I.~Klep, and S.~McCullough.
\newblock The matricial relaxation of a linear matrix inequality.
\newblock \emph{Math. Program.}, 138\penalty0 (1-2, Ser. A):\penalty0 401--445,
  2013.

\bibitem[Helton et~al.(2018)Helton, Mai, and Speicher]{HMS-realize}
J.~W. Helton, T.~Mai, and R.~Speicher.
\newblock Applications of realizations (aka linearizations) to free
  probability.
\newblock \emph{Journal of Functional Analysis}, 274\penalty0 (1):\penalty0
  1--79, 2018.

\bibitem[Helton et~al.(2019)Helton, Klep, McCullough, and Schweighofer]{HKMS}
J.~W. Helton, I.~Klep, S.~McCullough, and M.~Schweighofer.
\newblock Dilations, linear matrix inequalities, the matrix cube problem and
  beta distributions.
\newblock \emph{Mem. Amer. Math. Soc.}, 257\penalty0 (1232):\penalty0 vi+106,
  2019.

\bibitem[Horn and Johnson(1991)]{HornJohnson}
R.~A. Horn and C.~R. Johnson.
\newblock \emph{Topics in matrix analysis}.
\newblock Cambridge University Press, Cambridge, 1991.

\bibitem[Jury and Martin(2019)]{JM-freeSmirnov}
M.~T. Jury and R.~T.~W. Martin.
\newblock Unbounded operators affiliated to the free shift on the free {H}ardy
  space.
\newblock \emph{J. Funct. Anal.}, 277, 2019.

\bibitem[Jury et~al.(2020)Jury, Martin, and Shamovich]{NC-BSO}
M.~T. Jury, R.~T. Martin, and E.~Shamovich.
\newblock Blaschke-singular-outer factorization of free non-commutative
  functions.
\newblock arXiv:2001.04496, 2020.

\bibitem[Kaliuzhnyi-Verbovetskyi and Vinnikov(2009)]{KVV-ncrat}
D.~Kaliuzhnyi-Verbovetskyi and V.~Vinnikov.
\newblock Singularities of rational functions and minimal factorizations: the
  noncommutative and the commutative setting.
\newblock \emph{Linear Algebra and its Applications}, 430:\penalty0 869--889,
  2009.

\bibitem[Kaliuzhnyi-Verbovetskyi and Vinnikov(2014)]{KVV}
D.~Kaliuzhnyi-Verbovetskyi and V.~Vinnikov.
\newblock \emph{Foundations of free noncommutative function theory}, volume
  199.
\newblock American Mathematical Society, 2014.

\bibitem[Levick and Martin(2018)]{LM-dil}
J.~Levick and R.~Martin.
\newblock Matrix {N}-dilations of quantum channels.
\newblock \emph{Operators and Matrices}, 12:\penalty0 977--995, 2018.

\bibitem[Pascoe(2019)]{Pascoe}
J.~E. Pascoe.
\newblock The outer spectral radius and dynamics of completely positive maps.
\newblock arXiv:1905.09895, 2019.

\bibitem[Popescu(1991)]{Pop-vN}
G.~Popescu.
\newblock von {N}eumann inequality for {$(B({\scr H})^n)_1$}.
\newblock \emph{Math. Scand.}, 68\penalty0 (2):\penalty0 292--304, 1991.

\bibitem[Popescu(1995)]{Pop-multi}
G.~Popescu.
\newblock Multi-analytic operators on {F}ock spaces.
\newblock \emph{Mathematische Annalen}, 303:\penalty0 31--46, 1995.

\bibitem[Popescu(2006)]{Pop-freeholo}
G.~Popescu.
\newblock Free holomorphic functions on the unit ball of {$B (\mathcal{H})
  ^n$}.
\newblock \emph{J. Funct. Anal.}, 241:\penalty0 268--333, 2006.

\bibitem[Popescu(2014)]{PopRota}
G.~Popescu.
\newblock Similarity problems in noncommutative polydomains.
\newblock \emph{Journal of Functional Analysis}, 267:\penalty0 4446--4498,
  2014.

\bibitem[Richter(2019)]{Richter}
S.~Richter.
\newblock Free outer functions in complete {P}ick spaces.
\newblock Conference Presentation, 2019.

\bibitem[Salomon et~al.(2020)Salomon, Shalit, and Shamovich]{SSS}
G.~Salomon, O.~M. Shalit, and E.~Shamovich.
\newblock Algebras of noncommutative functions on subvarieties of the
  noncommutative ball: the bounded and completely bounded isomorphism problem.
\newblock \emph{J. Func. Anal.}, 278, 2020.

\bibitem[Shamovich(2018)]{Selfmaps}
E.~Shamovich.
\newblock On fixed points of self maps of the free ball.
\newblock \emph{J. Funct. Anal.}, 275\penalty0 (2):\penalty0 422--441, 2018.

\bibitem[Vol{\v{c}}i{\v{c}}(2017)]{Volcic}
J.~Vol{\v{c}}i{\v{c}}.
\newblock On domains of noncommutative rational functions.
\newblock \emph{Linear Algebra and its Applications}, 516:\penalty0 69--81,
  2017.

\end{thebibliography}

\vfill

\Addresses

\end{document}